\documentclass[11pt, reqno]{amsart}
\usepackage{amsmath}
\usepackage{amsthm}
\usepackage{amssymb}
\usepackage{amscd} 
\usepackage{amsbsy}
\usepackage{epsfig}
\usepackage{graphicx}
\usepackage{psfrag}
\usepackage{hyperref}

\theoremstyle{plain}
\newtheorem{theorem}{Theorem}[section]
\newtheorem{remark}[theorem]{Remark}
\newtheorem{proposition}[theorem]{Proposition}
\newtheorem{lemma}[theorem]{Lemma}
\newtheorem{corollary}[theorem]{Corollary}

\theoremstyle{definition}

\newfont\bbf{msbm10 at 12pt}
\def\eps{\varepsilon}
\def\R{{\mathbb R}}

\def\N{{\mathbb N}}
\def\Z{{\mathcal Z}}

\def\B{{\mathcal B}}
\def\P{{\mathcal P}}

\def\L{{\mathcal L}}

\def\cI{{\mathcal I}}
\def\es{{\emptyset}}
\def\sm{\setminus}

\def\Icont{I_{\small cont}}
\def\Ycont{Y_{\small cont}}

\def\diam{\mbox{\rm diam} }
\def\orb{\mbox{\rm orb}}

\def\le{\leqslant}
\def\ge{\geqslant}

\newcommand{\I}{\mathring{I}}
\newcommand{\f}{\mathring{f}}

\newcommand{\Lp}{\mathcal{L}}

\newcommand{\vf}{\varphi}
\newcommand{\ve}{\varepsilon}

\newcommand{\hLp}{\mathring{\Lp}}

\newcommand{\beq}{\begin{equation}}
\newcommand{\eeq}{\end{equation}}

\numberwithin{equation}{section}

\begin{document}
\author[H. Bruin]{Henk Bruin}
\address{Henk Bruin\\ Faculty of Mathematics \\
Vienna University\\
Oskar Morgensternplatz 1, Vienna 1090 \\
Austria}
\email{\href{mailto:henk.bruin@univie.ac.at}{henk.bruin@univie.ac.at}}
\urladdr{\url{http://www.mat.univie.ac.at/~bruin/}}

\author[M.F. Demers]{Mark F. Demers}
\address{Mark F. Demers\\ Department of Mathematics \\
Fairfield University\\
Fairfield, CT 06824 \\
USA}\email{\href{mailto:mdemers@fairfield.edu}{mdemers@fairfield.edu}}
\urladdr{\url{http://faculty.fairfield.edu/mdemers}}

\author[M. Todd]{Mike Todd}
\address{Mike Todd\\ Mathematical Institute\\
University of St Andrews\\
North Haugh\\
St Andrews\\
KY16 9SS\\
Scotland} \email{\href{mailto:m.todd@st-andrews.ac.uk}{m.todd@st-andrews.ac.uk}}
\urladdr{\url{http://www.mcs.st-and.ac.uk/~miket/}}

\title{Hitting and escaping statistics: mixing, targets and holes}
\date{\today}

\thanks{
MD was partially supported by NSF grant DMS 1362420.  This project was started 
as part of an RIG grant through ICMS, Scotland.  The authors would like to thank ICMS
for its generous hospitality.  They would also like to thank the ICERM Semester Program on Dimension 
and Dynamics and the ESI Thematic Programme Mixing Flows and Averaging Methods where part of this work was carried out.}


\begin{abstract}
There is a natural connection between two types of recurrence law: hitting times to shrinking targets, and hitting times to a fixed target (usually seen as escape through a hole).  We show that for systems which mix exponentially fast, one can move through a natural parameter space from one to the other.  On the other hand, if the mixing is subexponential, there is a phase transition between the hitting times law and the escape law.
\end{abstract}

\maketitle

\section{Introduction}

This work is motivated by the natural connection between escape rates and hitting times. 
The existence of an exponential Hitting Time Statistics (HTS) law, which is a recurrence law to shrinking targets, is a rather soft condition: in all cases we are aware of, all one requires is mixing, with no rates necessary.  However, under some mixing conditions, 
good error bounds can be derived (see e.g.\ \cite{FreFreTod15}) which mean that we can change the scaling in that law and still derive a non-degenerate limit law. If the mixing is exponential, the scaling can be changed to recover an escape rate to a fixed hole/target.
In this paper we explore a parameter space which takes us between the escape rate case and the hitting time case. Under exponential mixing we can go between these laws in a non-degenerate way. A phase transition occurs when one leaves the hitting time setting and heads towards the escape case whenever the system is subexponentially mixing.
In this paper we address such transitions in the case of stretched exponential, super-polynomial
and polynomial rates of mixing.

\subsection{Hitting times, escape rates, and between}
\label{sec:hitting}

Given a dynamical system
$f:X \circlearrowleft$ preserving an ergodic probability measure $\mu$, one can consider first entry times to a sequence of subsets $(U_r)_r$ with
$U_r$ shrinking to a given point $z$ as $r \to 0$.
Letting $\tau_r$ be the \emph{first hitting time} to $U_r$, i.e., 
$$\tau_r(x):=\inf\left\{n\ge 1: f^n(x) \in U_r\right\},$$ one can ask how the quantity $\mu(\tau_r > t)$ depends asymptotically on both $r$ and $t$ (for a fixed $z$).  
To derive a HTS law, one scales the time via $t = s/\mu(U_r)$ for some $s \in \mathbb{R}^+$ and considers the limit
\begin{equation*}
\lim_{r \to 0} \mu(\tau_r > s/\mu(U_r))   
\label{eq:HTS no log}
\end{equation*}
For a large range of dynamical systems it is known that this limit is $e^{-s}$ for $\mu$-a.e. centre $z$.  So we obtain an expression which is more convenient in this work:
\begin{equation}
\lim_{r \to 0} - \frac{1}{s}\log \mu(\tau_r > s/\mu(U_r)) = 1 \qquad \mbox{for $\mu$-a.e. $z$} .
\label{eq:HTS}
\end{equation}
There is a wealth of literature on this topic, but here we just refer to the reviews \cite{Hay13} and \cite[Chapter 5]{Exbook} and note that we only require very basic mixing properties for \eqref{eq:HTS}; for example, for multimodal maps of the interval, if there is an absolutely continuous invariant measure (with no mixing requirement), this law holds \cite{BruTod09}.

From the point of view of open systems, one declares $U_r$ to be a (fixed) hole and considers 
any point entering $U_r$ to be annihilated from the system.  In contrast to hitting times
(where $\tau_r(x) \geq 1$ a.s.~also for $x \in U_r$), in an open system
a point $x \in U_r$ is not allowed to exit $U_r$.  Thus the escape time $e_r(x)$ satisfies
$e_r(x) = \tau_r(x)$ if $x \notin U_r$ and $e_r(x) = 0$ if $x \in U_r$.  However, an essential
connection between the two is given by,
\[
\big\{ x \in X : \tau_r(x) = t \big\} = f^{-1}\left( \big\{ x \in X : e_r(x) = t-1 \big\} \right), \qquad \mbox{for all $t \ge 1$.}
\]
Due to the invariance of $\mu$, the escape rate can be defined by the following
equivalent expressions,
\begin{equation}
\lim_{t \to \infty} - \frac 1t \log \mu (e_r > t-1) = \lim_{t \to \infty} - \frac{1}{t} \log \mu(\tau_r > t)  \label{eq:esc rate}
\end{equation}
when this limit exists.  If the limit exists, we label it $-\log \lambda_r$ for reasons that will become clear later and consider the 
`derivative of the escape rate', expressed as the limit,
\begin{equation}
\lim_{r \to 0} \frac{- \log \lambda_r}{\mu(U_r)} = 1 \qquad \mbox{for $\mu$-a.e. $z$}, 
\label{eq:esc r deriv}
\end{equation}
which has been proved for certain exponentially mixing systems \cite{BY,KL zero}. 
We are not aware of examples where  the limit in \eqref{eq:esc rate} exists
(in the exponentially mixing setting), but \eqref{eq:esc r deriv} fails. 
Naturally if the system is subexponentially mixing, then \eqref{eq:esc rate} 
should be degenerate and so \eqref{eq:esc r deriv} fails, see \cite{DemFer}. 
One expects (see e.g.\ \cite{FerPol12, FreFreTod13}) that for periodic points $z$, 
the limit will be some number in $(0,1)$ which can be expressed in terms of the relevant potential; 
if $f$ is continuous, for all other points the limit should be 1.
The recent work \cite{PolUrb} extends this point of view to a wide variety of conformal
systems via symbolic dynamics.

Both of the limits \eqref{eq:HTS} and \eqref{eq:esc r deriv} can be seen as special limiting cases of the expression
\begin{equation}
\frac{1}{\mu(U_r)} \frac{-1}{t} \log \mu(\tau_r > t), \label{eq: general lim}
\end{equation}
where the open system perspective takes first the limit $t \to \infty$ then $r \to 0$, while the hitting time perspective takes the 
`diagonal limit' $r \to 0$ with $t = s/\mu(U_r)$.

Once one views this expression in the two-dimensional parameter space $(r,t)$,
one can naturally ask questions regarding convergence along various paths through
this parameter space.
Setting $t = s \mu(U_r)^{-\alpha}$ for some $\alpha, s \in (0, \infty)$, we formulate the generalised limit 
\begin{equation}\label{eq:L}
L_{\alpha, s}(z) :=\lim_{r \to 0}  \frac{-1}{s \mu(U_r)^{1-\alpha}} \log \mu(\tau_r > s\mu(U_r)^{-\alpha}),
\end{equation}
if the limit exists.
With this formulation, the case $\alpha = 1$ coincides with the diagonal limit formulated above for hitting time statistics.
Additionally, $\alpha = \infty$ can be thought of as coinciding with the derivative of the escape rate
\eqref{eq:esc r deriv}
(where $t \to \infty$ as $r$ is held fixed), 
while $\alpha = 0$ can be thought of as the
reversed order of limits,
  \[
\lim_{t \to \infty} \lim_{r \to 0} \frac{-1}{t \mu(U_r)} \log \mu(\tau_r > t).
\]

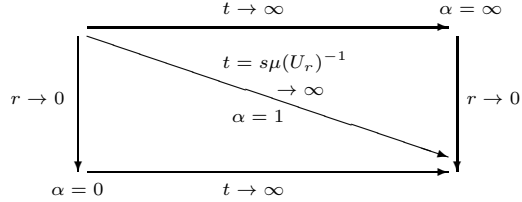
\begin{figure}[h]
\unitlength=6mm
\begin{picture}(10,4)(0,-0.2) \let\ts\textstyle
\put(0,0){\vector(1,0){8}}  \put(3,-0.5){\tiny $t \to \infty$}
\put(0,3.2){\vector(1,0){8}}  \put(3,3.5){\tiny$t \to \infty$}
\put(0,3){\vector(3,-1){8}}  \put(3,2.3){\tiny$t = s\mu(U_r)^{-1}$}
 \put(4.2,1.7){\tiny$\to \infty$}
\put(-0.2,3){\vector(0,-1){3}}  \put(-1.7,1.5){\tiny$r \to 0$}
\put(8.2,3){\vector(0,-1){3}}  \put(8.4,1.5){\tiny$r \to 0$}
\put(7.8, 3.5){\tiny$\alpha = \infty$}\put(-0.8, -0.5){\tiny$\alpha = 0$}
\put(3.2, 1.1){\tiny$\alpha = 1$}
\end{picture}
\caption{Different paths of taking the limit
$r\to 0$, $t \to \infty$, with $t = s \mu(U_r)^{-\alpha}$.}
\label{fig:square}
\end{figure}

\begin{remark}\label{rem:alpha}
For every $\alpha \in [0,1)$, and supposing $\mu(\tau_r > t) > 0$
for all $t$,
$L_{\alpha, s}(z) \in [0,1]$,
provided it exists.  Indeed, for $t = s \mu(U_r)^{-\alpha}$, we have
\begin{eqnarray*}
0 &\le& \frac{-\log \mu(\tau_r > t)}{s \mu(U_r)^{1-\alpha}} = 
\frac{-\log(1-\mu(\tau_r \le t))}{s \mu(U_r)^{1-\alpha}} \\
&=&
\frac{-\log(1-\mu(\cup_{j=0}^{t-1} f^{-j}(U_r)))}{s \mu(U_r)^{1-\alpha}}
\le \frac{-\log(1-t \mu(U_r))}{s \mu(U_r)^{1-\alpha}}
 \\
&=& \frac{-\log( 1-s \mu(U_r)^{1-\alpha} )}{s \mu(U_r)^{1-\alpha}} \to 1
\quad \text{as } \mu(U_r) \to 0.
\end{eqnarray*}
Therefore, any limit point belongs to $[0,1]$.
The calculation above also implies that, when the limit exists,
for $\alpha < 1$,
\begin{equation}\label{eq:mutau}
\lim_{r \to 0} \frac{- \log \mu(\tau_r > s \mu(U_r)^{-\alpha})}{s\mu(U_r)^{1-\alpha}} = \lim_{r \to 0} \frac{\mu(\tau_r \le s\mu(U_r)^{-\alpha})}{s\mu(U_r)^{1-\alpha}}.
\end{equation}
\end{remark}

\subsection{Brief summary of results}

Our main results are, roughly speaking, that if the system behaves well and is exponentially mixing, then $L_{\alpha, s}(z)$ exists 
for all $\alpha$;  it can be written in terms of the periodic behaviour if $z$ is periodic, and  $L_{\alpha, s}(z)=1$ otherwise 
(Theorem~\ref{thm:exp}).  
On the other hand, if the system is slower than stretched exponentially mixing then $L_{\alpha, s}(z)=0$ for $\alpha>1$.  
If the system is (exactly) stretched exponentially mixing then there exists an $\alpha_0>1$ depending on the mixing rate 
so that we have the same result as for the exponential case if $\alpha<\alpha_0$, and $L_{\alpha, s}(z)=0$ for $\alpha>\alpha_0$ 
(Theorem~\ref{thm:induce}).  The latter results employ an inducing argument and a large deviations law 
(either exponential, stretched exponential or polynomial).  Our examples using inducing schemes require good large deviations 
of the inducing time, with various types of tail.  

We remark that the existence of $L_{\alpha,s}(z)$ for $\alpha \neq 1$ is more delicate than 
for $\alpha =1$ and gives additional information about the distribution of $\tau_r$.  For example,
in the generic case, when $\alpha =1$, one obtains that
$\mu(\tau_r > s\mu(U_r)^{-1}) \to e^{-s}$ as $r \to 0$, but the rate of convergence does not appear.
By contrast, when $\alpha \neq 1$, the limit of $\mu(\tau_r > s\mu(U_r)^{-\alpha})$ is always
either $1$ (for $\alpha <1$) or $0$ (for $\alpha > 1$), and $L_{\alpha,s}(z)$ captures the
exponential rate at which this convergence occurs.  This rate provides information about the
tail distribution of $\tau_r$ for small $r$.  Again using the generic case as an example,
when $\alpha > 1$, $L_{\alpha,s}(z)=1$ implies
$\mu(\tau_r > s\mu(U_r)^{-\alpha}) = e^{-(1 \pm \ve) s\mu(U_r)^{1-\alpha}}$;
when $\alpha < 1$, $L_{\alpha,s}(z)=1$ implies
$\mu(\tau_r \le s\mu(U_r)^{-\alpha}) = (1 \pm \ve) s \mu(U_r)^{1-\alpha}$,
due to \eqref{eq:mutau}. 

The paper is organised as follows.
In Section~\ref{sec:exp mix}, we consider interval maps with good spectral properties; namely, 
that an associated family of transfer operators has a spectral gap.   Our results are formulated abstractly, 
but in Section~\ref{sec:examples} we give specific examples including Lasota-Yorke maps, Gibbs-Markov maps and the Gauss map.  
In Section~\ref{sec:inducing}, we consider systems where a well-chosen first return map has good properties and show how 
the tail of the first return time affects the limits $L_{\alpha, s}$.  
Again we formulate our results abstractly and then provide examples 
in Section~\ref{sec:app} to a variety of maps, including generalised Farey maps,
several classes of unimodal maps and Young towers. 
In the appendix we (re)prove two technical results used in Sections~\ref{sec:exp mix} and \ref{sec:inducing}.

\medskip
\noindent
{\bf Notation.}  We will use the following notational conventions throughout the paper without further mention.  $A = C^{\pm 1} B$ means there exists $C \ge 1$ such that $C^{-1} B \le A \le C B$;
similarly, $A = (1 \pm \ve) B$ means $(1- \ve)B \le A \le (1+ \ve)B$.  We write $A \sim B$ if
$\lim A/B = 1$, where the parameter in which the limit is taken is clear by context (usually it is
$r \to 0$).


\section{Exponentially mixing case}

\label{sec:exp mix}

In this section, we consider a piecewise continuous map of the unit interval $f: I \circlearrowleft$, 
with countably many intervals of monotonicity.  Our assumptions will be general enough to allow both traditional piecewise expanding maps as well as more general
Gibbs-Markov maps with contracting potentials.  We will then prove our results
regarding $L_{\alpha,s}(z)$, defined in \eqref{eq:L},
with respect to equilibrium states for these potentials. We will make assumptions on the map ((F1)--(F4) below as well as \textbf{(U)}, and where appropriate \textbf{(P)}) which imply the conditions of Rychlik \cite{Rychlik} as well as giving a form of the Lasota-Yorke inequality
needed in Proposition~\ref{prop:LY}. 

Assume that there exists a countable collection of maximal intervals $\Z = \{ Z_i \}_i$, $Z_i \subset I$,
with disjoint interiors, such that $f$ is continuous and strictly monotonic on each $Z_i$.  
We set $D = I \setminus \cup_{i} \mbox{int}(Z_i)$.

We assume that there exists a (nonatomic) Borel probability measure $m_\vf$ such that $m_\vf(D) =0$,
which is conformal with
respect to a potential
$\vf : I \to \mathbb{R}$, i.e., $dm_\vf/d(m_\vf \circ f) = e^\vf$. 
The associated transfer operator acting on $L^1(m_\vf)$ is
\[
\Lp_\vf \psi(x) = \sum_{y \in f^{-1}x} \psi(y) e^{\vf(y)} , \qquad \forall \psi \in L^1(m_\vf).
\]
We will study the action of $\Lp_\vf$ on functions of bounded variation.
To that end, recall that the variation of a function $\psi$ on an interval
$J$ is defined by
\[
\bigvee_J \psi = \sup \left\{ \sum_{i=0}^{k-1} |\psi(x_{i+1}) - \psi(x_i)| \; : \; x_0 < x_1 < \cdots < x_k,
\ x_i \in J, \forall i \le k \right\} , 
\]
where the supremum is taken over all finite sets $\{ x_i \}_i \subset J$.

Let $S_n \vf = \sum_{i=0}^{n-1} \vf \circ f^i$. 
We set $\vf|_D = - \infty$ and assume the potential $\vf$ satisfies the following regularity properties:
\begin{enumerate}
  \item[(F1)] $\exists C_d >0 $ s.t. $ |e^{S_n\vf(x) - S_n\vf(y)} -1| \le C_d|f^nx - f^ny|$ whenever
  $f^ix, f^iy$ lie in the same element of $\Z$ for $i = 0, 1, \ldots, n-1$;
  \item[(F2)] $\sum_{Z \in \Z} \sup_Z e^{\vf} < \infty$;
  \item[(F3)] $\exists n_0 \in \N$ such that $\sup_I e^{S_{n_0} \vf} < \inf_{I\setminus D} \Lp_\vf^{n_0} 1$;
  \item[(F4)]  for each interval $J \subset I\setminus D$, $\exists N = N(J)$ s.t. 
  $\inf_{I \setminus D} \Lp_\vf^N 1_J > 0$, where $1_J$ is the indicator function of the set $J$.
\end{enumerate}  
 
Due to the existence of the conformal measure $m_\vf$, we have
$\int \Lp_\vf^n 1 \, dm_\vf = \int 1\, dm_\vf = 1$, so that
$\inf_{I\setminus D} \Lp_\vf^n 1 \le 1$ for each $n \in \N$.  Thus by (F3), $\sup_I e^{S_{n_0}\vf} < 1$.
Then since $\sup_I e^{S_n \vf}$ is submultiplicative,
\begin{equation}
\label{eq:n1 def}
\exists n_1 \in \N \mbox{ such that } (2+2C_d) \sup_I e^{S_{n_1}\vf} < 1,
\end{equation}
where $C_d$ is from (F1).

Now fix $z \in I$ and for $r_0 > 0$, define $( U_r )_{r \in (0, r_0)}$ to be a 
family of intervals such that $\diam(U_r) \to 0$ as $r \to 0$, and $\cap_{r} U_r = \{ z \}$.  
From the point of view of open systems, for each
$r$, we define the map with hole $U_r$ and its iterates by, $\f_r^n = f^n |_{\I_r^{n-1}}$, where 
$\I_r^0 = I \setminus U_r$ and $\I_r^n = \cap_{i =0}^{n} f^{-i}(I \setminus U_r)$.

Let $\cI_r^n$ denote the intervals of monotonicity for $\f_r^n$.
We assume the following uniform large images condition for $\f_r^{n_1}$ on the sequence $(U_r)_r$.
\begin{itemize}
  \item[{\bf (U)}]  There exists $c_0 > 0$ such that 
  $$\inf_{r\in [0,r_0]} \inf \{ m_\vf(\f^{n_1}_rJ) : J \in \cI_r^{n_1} \} \ge c_0,$$ 
\end{itemize}
where $n_1$ is from \eqref{eq:n1 def}.

As we shall show in Section~\ref{prelim}, under assumptions (F1)-(F4), $\Lp_\vf$
admits a unique invariant measure $\mu_\vf$, absolutely continuous with respect to $m_\vf$, whose density $g_0$ is of bounded variation and is bounded away 
from $0$.   Note that $\mu_\vf$ can also be characterised as an equilibrium state for $\vf$.  That is, for the variational pressure $P(\vf):=\sup\left\{h(\mu)+\int\vf~d\mu\right\}$ where the supremum is taken over all $f$-invariant probability measures, $\mu_\vf$ is an equilibrium state since it satisfies $h(\mu_\vf)+\int\vf~d\mu_\vf=P(\vf)$.  Moreover, since we can also think of $P(\vf)$ as the $\log$ of the leading eigenvalue of $\L_\vf$, our assumptions here give $P(\vf)=0$.

In the case when $z$ is periodic of prime period $p$,
we shall need the following additional condition.

\begin{itemize}
 \item[{\bf (P)}] The density $g_0$ is continuous at $z$ and $f^p$ is monotonic at $z$.
\end{itemize}
Let 
\begin{equation}\label{eq:Icont}
\Icont:=\{z\in I: f^k \text{ is continuous at } z \text{ for all } k\in \N\}.
\end{equation}
That $\mu_\vf(\Icont) =1$ follows from the assumption that $m_\vf(D) =0$.
The main result of this section is the following theorem.

\begin{theorem}
\label{thm:exp}
Let $(f, \vf)$ satisfy (F1)-(F4).
Fix $z \in \Icont$ and let $(U_r)_{r \in (0, r_0]}$ be a 
family of intervals such that $\lim_{r \to 0} \diam(U_r) = 0$ 
and $\cap_r U_r = \{ z \}$,
satisfying {\bf (U)}, and in the periodic case, {\bf (P)} as well.

Then, for any $s \in \mathbb{R}^+$ and $\alpha \in [0,\infty]$, taking $L_{\alpha,s}(z)$
with respect to the invariant measure $\mu_\vf$, we have
$$
L_{\alpha,s}(z) = \begin{cases}
       1, & \mbox{if $z$ is not periodic} \\
       1 - e^{S_p\vf(z)}, & \mbox{if $z$ is $p$-periodic},
\end{cases}
$$
where $p$-periodic means that the prime period of $z$ is $p$. 
\end{theorem} 

\begin{remark}
If $f$ is continuous, then $I=\Icont$, while in the context of hitting time statistics, the case $z\in I\sm \Icont\neq \es$ is addressed in \cite[Section 3.3]{AytFreVai15}.
\end{remark}

\begin{remark}\label{rem:alpha 0}
As will be clear from the proof of Theorem~\ref{thm:exp}, the case $\alpha=0$ holds
in great generality:  The proof in Section~\ref{ssec:alpha 0} requires neither (F2)-(F4) nor {\bf (U)}.   
In the non-periodic case we require only that $z\in \Icont$.
In the periodic case, we require (F1), {\bf (P)} and the fact that $\mu_\vf$ is absolutely
continuous with respect to $m_\vf$ with density bounded away from 0 at $z$.  
Alternatively, if it is known that $e^{S_p\vf}$ is continuous at $z$, then (F1) is not needed.
\end{remark}


\subsection{Preliminaries}
\label{prelim}

We begin by establishing some easy facts about the potential $\vf$.  Let 
$\Z^n = \bigvee_{i=0}^{n-1} f^{-i}(\Z)$ 
denote the maximal intervals on which $f^n$ is continuous and monotonic.

\begin{lemma}
\label{lem:pot}
Assuming (F1)-(F4), for all $n \ge 0$, the following hold:
\begin{itemize}
  \item[(a)]  $\sum_{Z \in \Z^n} e^{S_n \vf} < \infty$;
  \item[(b)]  for each $Z \in \Z$, $\bigvee_{Z} e^\vf \le C_d \sup_Z e^\vf$;
  \item[(c)]  $\bigvee_I e^{S_n\vf} < \infty$. 
\end{itemize}
\end{lemma}

\begin{proof}
(a) follows from a standard inductive argument using (F2).  

(b) follows from (F1) since
 $|e^{\vf(x_{i+1})} - e^{\vf(x_i)}| \le C_d e^{\vf(x_i)}$ for any set $\{ x_i \}_{i=1}^k \subset Z$.
 
For $n=0$, (c) follows from (b) and (F2).  Note that setting $\vf|_D = -\infty$ only adds
a term bounded by the series in (F2) to the variation.  For $n \ge 1$, the argument again
follows from a standard induction.
\end{proof}

Potentials satisfying the above properties in addition to (F3) are called
{\em contracting potentials} in the literature (see for example, \cite{Rychlik, LSV1}),
while (F4) is called the {\em covering property.}
However, we require
(F1) in order to obtain the stronger form of Lasota-Yorke inequalities in Proposition~\ref{prop:LY},
which we shall need to apply perturbation theory to the open systems $(\f_r, U_r)$,
considering $U_r$ as a hole.\footnote{
For a way to relax condition
(F1) by requiring only a H\"older bound on distortion, 
see the application of Theorem~\ref{thm:exp} to the Gauss map in Section~\ref{gauss}.}
We will prove Theorem~\ref{thm:exp} using the fact that the transfer operators associated with both the closed and open systems
have spectral gaps and their spectral projectors vary in some uniform way with the size of the hole.

Let
$\B$ be the set of functions of bounded variation on $I$ equipped with the variation norm
$\| \psi \| = \bigvee_I \psi + |\psi|_1$, where $| \cdot |_1$ denotes the $L^1$-norm with respect to 
the conformal measure $m_\vf$.

Using Lemma~\ref{lem:pot} and \eqref{eq:n1 def}, the operator
$\Lp_\vf^{n_1}$ satisfies the assumptions of \cite[Theorem 1]{Rychlik}, so that $\Lp_\vf^{n_1}$
is quasi-compact as an operator on $\B$.  Now using the decomposition in 
\cite[Theorem 3]{Rychlik} and the covering property (F4), it follows that $\Lp_\vf^{n_1}$ has
a simple eigenvalue at $1$, and no other eigenvalue can have modulus 
$1$, i.e., $\Lp_\vf^{n_1}$
has a spectral gap.  Using again
Lemma~\ref{lem:pot}(b),(c), since $\Lp_\vf$ is a bounded operator on $\B$, it also has a spectral
gap.  This will be the starting point from which we will perturb.


\subsection{Notation and Initial Results for Open Systems}
\label{open}

In this section, we first summarise standard notation for open systems that we will use
throughout the paper.  We then proceed to prove the existence of a uniform spectral gap
for a family of associated transfer operators.

Recall that if we regard $U_r$ as a hole, then the set of points that has not entered $U_r$
by time $n$ is denoted $\I_r^n =\cap_{i=0}^n f^{-i}(I \setminus U_r)$, and the map corresponding to the open system is
simply the restriction $\f_r^n := f^n|_{\I_r^{n-1}}$.  Notice that by definition of the escape time
$e_r$ (see Section~\ref{sec:hitting}), we have $\{ e_r > n \} = \I_r^n$.

The transfer operator for the open system and its iterates are defined for $\psi \in L^1(m_\vf)$ by
\begin{equation}
\label{eq:open transfer}
\hLp_{\vf, U_r}^n \psi(x) = \sum_{y \in \f_r^{-n}x} \psi(y) e^{S_n\vf(y)} = \Lp_\vf^n (\psi 1_{\I_r^{n-1}})  ,
\end{equation}
for all $n \ge 1$.
Due to the conformality of the measure $m_\vf$, we have the following useful identity,
\begin{equation}
\label{eq:cov}
\int_I \hLp^n_{\vf, U_r} \psi \, dm_\vf = \int_I \Lp^n_\vf(\psi 1_{\I_r^{-n-1}}) \, dm_\vf
= \int_{\I_r^{n-1}} \psi \, dm_\vf .
\end{equation}
The importance of the above relation is the connection it provides
between the escape rate with respect to the measure $\psi dm_\vf$ and 
the spectral radius of $\hLp_{\vf, U_r}$ acting on functions of bounded variation (see Proposition~\ref{prop:LY}). 

Since we fix the potential $\vf$, for ease of notation and to emphasise the relationships
among the operators, in what follows we will denote
$\hLp_r := \hLp_{\vf, U_r}$ and $\Lp_0 := \Lp_{\vf}$.  Similarly, we denote by
$g_0$ the invariant density for $\Lp_0$, $m_0$ the conformal measure,
and $\mu_0 = g_0 m_0$ the invariant measure for the closed system.

Due to {\bf (U)}, we have the following set of uniform Lasota-Yorke 
inequalities for this family of operators.

\begin{proposition}
\label{prop:LY}
There exists $C_0>0$ and $\sigma < 1$ such that for any $\psi \in \B$, $r \in [0,r_0]$ and all $n \geq 0$,
\[
\begin{split}
\| \hLp_r^n \psi \| & \le C_0 \sigma^{-n} \| \psi \| + C_0 \int_{\I^{n-1}_r} |\psi| \, dm_0 , \\
| \hLp_r^n \psi |_1 & \le \int_{\I^{n-1}_r} |\psi | \, dm_0  .
\end{split}
\]
\end{proposition}

The proof is by now fairly standard, even in this generalised context.  Since our assumptions
and estimates necessarily differ
 from those appearing in the literature for closed systems (given that we must show uniformity of the constants
$C_0$ and $\sigma$ in the sequence $(U_r)$ as well as the fact that we require decay in the
$L^1$ term), we include the proof for completeness in the
appendix. 

It follows from Proposition~\ref{prop:LY}, the compactness of the unit ball of $\B$ in $L^1(m_0)$,
and the conformality of $m_0$ that
the spectral radius of $\hLp_r$ acting in $\B$
is at most one while its essential spectral radius
is bounded by $\sigma^{-1} <1$. Thus $\hLp_r$ is quasi-compact as an operator on $\B$,
as is $\Lp_0$.
In addition, defining the following perturbative norm, 
\[
||| \Lp_0 - \hLp_r ||| = \sup \{ |\Lp_0 \psi - \hLp_r \psi|_1 : \| \psi \| \le 1 \},
\]
we have the following bound.

\begin{lemma}
\label{lem:close}
$||| \Lp_0 - \hLp_r ||| \le m_0(U_r) \le C_1 \mu_0(U_r)$, where $C_1^{-1} = \mbox{essinf } g_0$.
\end{lemma}
\begin{proof}
The proof is immediate since if $\psi \in \B$ with $\| \psi \| \le 1$, we use the fact that
$m_0$ is $\vf$-conformal to estimate,
\[
\int |(\Lp_0 - \hLp_r)\psi | \, dm_0 = \int |\Lp_0(1_{U_r} \psi)| \, dm_0 \le |\psi|_\infty m_0(U_r),
\]
and the fact that $\mbox{essinf } g_0 > 0$ follows from (F4).
\end{proof}

\begin{corollary}
\label{cor:uniform gap}
There exists $r_1 \in (0, r_0]$ such that for all $r \in [0,r_1]$, the operators $\hLp_r$ have a uniform spectral gap on $\B$.  In particular, for $r>0$, 
there exist $\lambda_r < 1$, and linear operators $\Pi_r$, $R_r$, such that
\[
\hLp_r = \lambda_r \Pi_r + R_r,
\]
$\Pi_r^2 = \Pi_r$, $\Pi_r R_r = R_r \Pi_r = 0$ and the spectral radius of $R_r$ is at most
$\rho < \inf \{ \lambda_r : r < r_1\}$.  
The range of $\Pi_r$ is the span of a function $g_r \in \B$, satisfying
$\hLp_r g_r = \lambda_r g_r$, and normalised so that $\int g_r \, dm_0 =1$.

The above decomposition also holds for $r=0$ with $\lambda_0 = 1$.
\end{corollary}

\begin{proof}
$\Lp_0$ has a spectral gap by \cite{Rychlik} and the discussion following Lemma~\ref{lem:pot}.
It follows from Proposition~\ref{prop:LY}, Lemma~\ref{lem:close} and 
\cite[Corollary 1]{KL pert} that the spectra and spectral projectors of $\hLp_r$
and $\Lp_0$ outside the disk of radius $\sigma$ vary continuously in $\mu_0(U_r)$.
Thus for $r$ sufficiently small, $\hLp_r$ inherits a spectral gap from $\Lp_0$, 
and by continuity, the spectral gap is uniform in $r$, yielding the existence of 
$\rho < \inf \{ \lambda_r : r < r_1 \}$
in the statement of the corollary.  
\end{proof}

We proceed to the proof of Theorem~\ref{thm:exp}, first proving the special cases $\alpha = \infty$
and $\alpha = 0$, and then turning to the general case $\alpha \in (0,\infty)$.


\subsection{Proof of Theorem~\ref{thm:exp}: The case $\alpha = \infty$}

To address the case corresponding to $\alpha = \infty$, we must compute the double limit,
\[
\lim_{r \to 0} \lim_{t \to \infty} \frac{1}{\mu_0(U_r)} \frac{-1}{t} \log \mu_0 (\tau_r > t) .
\]
For fixed $r \in (0, r_1]$, the spectral gap provided by Corollary~\ref{cor:uniform gap}
implies that the
escape rate with respect to $\mu_0$ is $- \log \lambda_r$, i.e., 
\[
\lim_{t\to \infty} \frac 1t \log \mu_0(\tau_r > t) = \lim_{t \to \infty} \frac{1}{t} \log \mu_0 (\I^t_r) = \log \lambda_r ,
\]
where we have used \eqref{eq:esc rate} as well as the fact that $\{ e_r > t \} = \I^t_r$.
(Indeed, the escape rate is $-\log \lambda_r$ with respect to the measure $\psi m_0$ for
any density $\psi \in \B$ that is bounded away from $0$.)

In order to show the limit $r \to 0$ converges to the claimed value, we will use the results of
\cite{KL zero}.  To do this we must check the necessary conditions given there, listed
 as (A1)-(A7).  In our setting, (A1)-(A3) are immediately satisfied by the existence of a uniform spectral gap for the operators $\hLp_r$ and the accompanying spectral decomposition given by 
Corollary~\ref{cor:uniform gap}.

(A4) requires that we normalise $m_0(g_r) =1$, which we have done, and that
there exists $C_2>0$ such that $\sup_{r \in [0,r_1]} \| g_r \| \le C_2$, i.e., the conditionally
invariant densities are uniformly bounded in $\B$.  This follows from the uniform
Lasota-Yorke inequalities given by Proposition~\ref{prop:LY} applied to $g_r$:
\[
\lambda_r^n \| g_r \| = \| \hLp_r^n g_r \| \le C_0\sigma^n \| g_r \| + C_0 \int_{\I^n_r} g_r \, dm_0
= C_0\sigma^n \| g_r \| + C_0 \lambda_r^n,
\] 
where we have used the fact that $\int_{\I^n_r} g_r \, dm_0 = \int \hLp_r^n g_r \, dm_0$, by
conformality.  Since $\sigma < \lambda_r$ in the spectral gap regime, we 
let $n\to \infty$ and conclude that
$\| g_r \| \le C_0$ independently of $r \in [0, r_1]$.

(A5) requires that $\eta_r := \| m_0(\Lp_0 - \hLp_r) \| \to 0$ as $r \to 0$, where $\| \cdot \|$ is the
norm of the linear functional  $m_0(\Lp_0 - \hLp_r) : \B \to \R$.  This is precisely 
Lemma~\ref{lem:close}, since if $\psi \in \B$, we have
\[
| m_0((\Lp_0 - \hLp_r)(\psi))| \le \int |(\Lp_0 - \hLp_r) \psi | \, dm_0 \le \| \psi \| m_0(U_r) ,
\]
so that $\eta_r = m_0(U_r)$.

(A6) requires\footnote{\cite{KL zero} actually states this bound, and subsequent ones, in terms of
a more general quantity, $\Delta_r$; however, in the present context,
$\Delta_r = \int (\Lp_0 - \hLp_r) g_0 \, dm_0 = \int_{U_r} g_0 \, dm_0 = \mu_0(U_r)$, and we will
use this simpler expression in what follows.}
$\eta_r \cdot \| (\Lp_0 - \hLp_r) g_0 \| \le C_3 \mu_0(U_r)$, for some $C_3 > 0$.
This is satisfied since (as noted in Lemma~\ref{lem:close}), essinf $g_0 = C_1^{-1} > 0$.
Thus,
\begin{align*}
\eta_r \cdot \| (\Lp_0 - \hLp_r) g_0 \| & = m_0(U_r) \| \Lp_0 (1_{U_r} g_0 ) \|
\le C_1 \mu_0(U_r) \| \Lp_0 \| \| 1_{U_r} g_0 \| \\
& \le C_1 \| \Lp_0 \| (\| g_0 \| + 2 |g_0|_\infty) \mu_0(U_r)
\le 3C_1 C_0 \| \Lp_0 \| \mu_0(U_r),
\end{align*}
as required.

Finally, (A7) requires that the limit
\[
q_k := \lim_{r \to 0} q_{k,r} := 
\lim_{r \to 0} \frac{m_0((\Lp_0 - \hLp_r) \hLp_r^k (\Lp_0 - \hLp_r)(g_0))}{\mu_0(U_r)},
\]
exists for each integer $k \ge 0$.
Notice that by conformality and using the fact that $\Lp_0 - \hLp_r = \Lp_0(1_{U_r} \cdot)$, we have
\[
m_0((\Lp_0 - \hLp_r) \hLp_r^k (\Lp_0 - \hLp_r)(g_0))
= \int 1_{U_r} \circ f^{k+1} \cdot 1_{\I^{k-1}_r} \circ f \cdot 1_{U_r} \cdot g_0 \, dm_0 .
\]
The product of indicator functions in the above expression is equivalent to 
the indicator function of the set
\[
E^k_r = \{ x \in U_r : f^i(x) \notin U_r, i = 1, \ldots, k, \mbox{ and } f^{k+1}(x) \in U_r \}.
\]
So $q_{k,r} = \frac{\mu_0(E^k_r)}{\mu_0(U_r)}$.  

If $z$ is not periodic, recall that since $z\in \Icont$,  $f^k$ is continuous at $z$ for each $k \in \N$, so for fixed $k$, the set $E^k_r$ is empty for all $r$ sufficiently small.  Thus $q_k = 0$ for all
$k \ge 0$.  On the other hand, if $z$ is periodic with prime period $p$, then 
for sufficiently small $r$, $E^k_r$ is empty except when $k = p-1$.  In this case, we use the
monotonicity and continuity of $f^p$ at $z$ (assumption {\bf (P)})
to conclude that $f^p(E^{p-1}_r) = U_r$.  Let $f_1$ denote this branch of $f^p$.
Now the continuity
of $g_0$ at $z$ and (F1) yield,
\[
q_{p-1} = \lim_{r \to 0} \tfrac{1}{\mu_0(U_r)} \int_{E^{p-1}_r} g_0 \, dm_0  
= \lim_{r \to 0} \tfrac{1}{\mu_0(U_r)} \int_{U_r} e^{S_p\vf \circ f_1^{-p}} g_0 \circ f_1^{-p} dm_0
= e^{S_p\vf(z)} ,
\]
where we have used the fact that $f_1^{-p}(z) = z$.

Having verified conditions (A1)-(A7) in our setting, we conclude by
\cite[Theorem~2.1]{KL zero}, that for $z \in \Icont$, 
\begin{equation}\label{eq:exp formula}
\lim_{r \to 0} \frac{-1}{\mu_0(U_r)} \log \lambda_r =  
\begin{cases}
1, & \mbox{if $z$ is not periodic}\\
1 - e^{S_p\vf(z)}, & \mbox{if $z$ is $p$-periodic}. 
\end{cases}
\end{equation}


\subsection{Proof of Theorem~\ref{thm:exp}:  The case $\alpha = 0$}
\label{ssec:alpha 0}

Since we will not need the measures $\mu_r$ in this section, for simplicity,
we will denote $\mu_0$ simply by $\mu$ and $m_0$ by $m$.

We fix $t$ and consider the limit $\lim_{r \to 0} \mu(U_r)^{-1} \log \mu(\tau_r > t)$.  Note that the set
$\{ \tau_r \le t \} = \cup_{j=0}^{t} f^{-j}(U_r)$.  

{\em  Case 1:  Nonperiodic $z$}.
Assume that $z\in \Icont$ is not a periodic point for $f$. 
 Then we may choose $r$ sufficiently small that the sets $f^{-j}(U_r)$, $j = 0, \ldots, t$,
are pairwise disjoint.  Thus,
\[
\begin{split}
\lim_{r \to 0} \mu(U_r)^{-1} \log \mu(\tau_r > t) & = \lim_{r \to 0} \mu(U_r)^{-1} \log (1 - \mu(\tau_r \le t)) \\
& = \lim_{r \to 0} \mu(U_r)^{-1} \log (1 - (t+1)\mu(U_r))
= - (t+1) .
\end{split}
\]
Proceeding to the second limit, we complete the proof of this case,
\[
\lim_{t \to \infty} \lim_{r \to 0} - t^{-1} \mu(U_r)^{-1} \log \mu(\tau_r > t) = \lim_{t \to \infty} t^{-1}(t+1) = 1 .
\]

{\em Case 2: Periodic $z$}.  Fix $z$ of prime period $p$ for $f$, satisfying {\bf (P)}.
Choose $r$ sufficiently small that the sets $f^{-i}(U_r)$, $i = 0, \ldots, p-1$ are pairwise disjoint.  This choice forces
$f^{-i}(U_r) \cap f^{-j}(U_r) = \emptyset$ except when $i-j$ is a multiple of $p$.

Suppose $t \in \N$ satisfies $t = (k+1)p -1$ for some $k \ge 0$.  Then
\[
\{ \tau_r \le t \} = \cup_{i=0}^k \cup_{j=0}^{p-1} f^{-ip-j}(U_r) .
\]
Note that by the above observation regarding when two pre-images of $U_r$ may intersect, we conclude that the sets in the union above
are disjoint for distinct $j$.  Thus,
\[
\mu(\tau_r \le t) = \sum_{j=0}^{p-1} \mu(\cup_{i=0}^k f^{-ip-j}(U_r)) = p \mu(\cup_{i=0}^k f^{-ip}(U_r))  .
\]

To estimate the measure of the remaining set, we prove the following lemma.

\begin{lemma}\label{lem:per}
Let $z$ be a point of continuity of $g_0$ of prime period $p$. 
For $\ve >0$, let $U_r$ be a sufficiently small
neighbourhood of $z$ with diam$(U_r) < \ve$
such that $f^p$ is monotonic on $U_r$ and for each $x \in U_r$,  
$|g_0(x) - g_0(z)| \le \ve$. 
If $k$ is such that $U_r, U_r \cap f^{-p}(U_r), \ldots, U_r \cap f^{-kp}(U_r)$ forms a decreasing sequence of sets, then
\[
\mu(\cup_{i=0}^k f^{-ip}(U_r)) = \mu(U_r) ( k+1 - k e^{S_p\vf(z)}(1 \pm \bar{C}\ve)),
\]
where $\bar{C} = C_d + C_1$, $C_d$ is from (F1) and $C_1$ is from Lemma~\ref{lem:close}.
\end{lemma}

\begin{proof}
Write $V_r = U_r \cap f^{-p}(U_r)$.
The proof goes by induction on $k$.  For $k=1$, we have
\begin{equation}
\label{eq:base case}
\begin{split}
\mu(U_r \cup f^{-p}(U_r)) & = \mu(U_r) + \mu(f^{-p}(U_r)) - \mu(V_r) \\
& = 2\mu(U_r) - \mu(U_r \cap f^{-p}(U_r)).
\end{split}
\end{equation}

Letting $f_1^p$ denote the branch of $f^p$ mapping $V_r$ onto $U_r$ monotonically,
\begin{align*}
\mu(V_r)  &= \int_{U_r \cap f^{-p}(U_r)} g_0 \, dm = 
\int_{U_r} g_0 \circ f_1^{-p} \, e^{S_n\vf \circ f_1^{-p}} \, dm \\
& = \mu(U_r) e^{S_p\vf(z)} + \int_{U_r} g_0 \circ f_1^{-p} e^{S_p\vf \circ f_1^{-p}}
(1- e^{S_p\vf(z) - S_p \vf \circ f_1^{-p}}) \, dm  \\
& \qquad + e^{S_p\vf(z)} \int_{U_r} (g_0 \circ f_1^{-p} - g_0) \, dm 
\end{align*} 
Using (F1), the first integral on the right hand side is bounded by
\[
e^{S_p\vf(z)} C_d \mbox{diam}(U_r) \int_{U_r} g_0 \circ f_1^{-p} e^{S_p\vf \circ f_1^{-p}} \, dm
\le e^{S_p\vf(z)} C_d \mbox{diam}(U_r) \mu(U_r),
\]
where we have changed variables again for the last inequality.  The second integral
on the right hand side is bounded by,
\[
e^{S_p\vf(z)} \ve m(U_r) \le e^{S_p\vf(z)} C_1 \ve \mu(U_r),
\]
where $C_1$ is from Lemma~\ref{lem:close}.
Putting these estimates together and using the fact that diam$(U_r) < \ve$, we obtain,
\begin{equation}
\label{eq:translate}
\mu(V_r) = (1 \pm \bar{C} \ve) e^{S_p\vf(z)} \mu(U_r),
\end{equation}
where $\bar{C} = C_d + C_1$.

Plugging this into \eqref{eq:base case} yields the lemma for $k=1$.

Now suppose the statement holds for $k$ and consider the set,
\[
\cup_{i=0}^{k+1} f^{-ip}(U_r) = f^{-(k+1)p}(U_r) \cup ( \cup_{i=0}^k f^{-ip}(U_r)) =: f^{-(k+1)p}(U_r) \cup A_k .
\]

We claim that any intersection between $f^{-(k+1)p}(U_r)$ and $A_k$ necessarily belongs to $f^{-kp}(U_r)$.  To see this,
suppose $x \in f^{-(k+1)p}(U_r) \cap f^{-jp}(U_r)$ for some $j \le k$.  
Then $f^{jp}(x) \in U_r \cap f^{-(k+1-j)p}(U_r)$, 
which necessarily remains
in $U_r$ for the next $k+1-j$ iterates of $f^p$, due to the nested property of the 
sets $U_r \cap f^{-ip}(U_r)$.
In particular, $f^{(k-j)p}(f^{jp}x) \in V_r$.  Thus $x \in f^{-kp}(U_r)$.

Using this fact about intersection as well as \eqref{eq:translate}, we now estimate,
\[
\begin{split}
\mu(\cup_{i=0}^{k+1} f^{-ip}(U_r)) & = \mu( f^{-(k+1)p}(U_r)) + \mu(A_k) - \mu\left(f^{-(k+1)p}(U_r) \cap f^{-kp}(U_r)\right) \\
& = \mu(U_r) + \mu(A_k) - \mu(U_r \cap f^{-p}(U_r)) \\
& = \mu(U_r) + \mu(A_k) - (1 \pm \bar{C} \ve)\mu(U_r) e^{S_p\vf(z)} ,
\end{split}
\]
and the lemma is proved using the inductive hypothesis on $\mu(A_k)$.
\end{proof}

Using the lemma, we may estimate
\[
\begin{split}
\frac{1}{\mu(U_r)} & \log \mu(\tau_r >t) = \frac{1}{\mu(U_r)} \log \left(1 - p \mu( \cup_{i=0}^k f^{-ip}(U_r))\right) \\
& = \frac{1}{\mu(U_r)} \log \left( 1 - \mu(U_r)p \left[k+1 - k(1\pm \ve)e^{S_p\vf(z)} \right] \right) \\
& \xrightarrow[r \to 0]{} - \left[pk + p - pk(1\pm \bar{C} \ve) e^{S_p\vf(z)} \right].
\end{split}
\]
Now dividing by $-t$ and taking the limit as $t \to \infty$ completes the proof of the periodic case, up to an error
$\pm \ve e^{S_p\vf(z)}$.  Since $\ve$ was arbitrary, the case is proved
for $t$ of the form $(k+1)p-1$.  

For more general $t = kp + \ell$, for some $\ell = 0, \ldots, p-1$,
we have 
\[
\mu(\tau_r > (k+1)p -1) \le \mu(\tau_r > t) \le \mu(\tau_r > kp-1),
\]
and since the upper and lower bounds yield the same limit as $k \to \infty$, the limit for general $t$
exists and has the same value.


\subsection{Proof of Theorem~\ref{thm:exp}:  The case $\alpha \in (0, \infty)$} 
\label{ssec:deriv}

Fix $z \in I$, $\alpha \in (0, \infty)$, and a sequence of intervals $(U_r)_{r \in r_0}$ satisfying {\bf (U)}.  
If $z$ is periodic, we also assume {\bf (P)}.
Let $t = s \mu_0(U_r)^{-\alpha}$ for some $s \in \mathbb{R}^+$.  We must consider the limit,
\[
\lim_{r \to 0} \frac{1}{s \mu_0(U_r)^{1-\alpha}} \log \mu_0( \tau_r > s \mu_0(U_r)^{-\alpha}) .
\]
As in Section~\ref{prelim}, there exists $r_1>0$ such that all associated
transfer operators $\hLp_r$ have a uniform spectral gap on $\B$ for all $r \in [0, r_1]$.

To simplify notation, set $k_r = \lfloor s \mu_0(U_r)^{-\alpha} \rfloor$.  Notice that,
\[
\begin{split}
\mu_0(\tau_r > k_r) & = \int_{\I^{k_r}_r} g_0 \, dm_0 = \int \hLp_r^{k_r+1} g_0 \, dm_0  \\
& = \lambda_r^{k_r+1} \int \lambda_r^{-k_r-1} \hLp_r^{k_r+1} (g_0 - g_r) \, dm_0 + \lambda_r^{k_r+1} \int g_r \, dm_0,
\end{split}
\]
where $g_r$ is the unique normalised conditionally invariant density corresponding to $\lambda_r$
from Corollary~\ref{cor:uniform gap}.
Thus
\begin{equation}
\label{eq:alpha split}
\log \mu_0(\tau_r > k_r) = (k_r+1) \log \lambda_r +  \log \Big(1 + \int \lambda_r^{-k_r-1} \hLp_r^{k_r+1} (g_0 - g_r) \, dm_0 \Big).
\end{equation}
Notice that the first term above, when divided by $s\mu_0(U_r)^{1-\alpha}$, is simply $\mu_0(U_r)^{-1} \log \lambda_r$ (up to integer part) and thus
converges as $r \to 0$ to the required limit by \eqref{eq:exp formula}, which depends on
$z$.  

It remains to show that the second term in \eqref{eq:alpha split} converges to zero after division by $\mu_0(U_r)^{1-\alpha}$.
Using Corollary~\ref{cor:uniform gap}, we may decompose the transfer operator as 
$\hLp_r = \lambda_r \Pi_r + R_r$, where as before, $\Pi_r$ is the projection onto the eigenspace spanned by $g_r$ and the spectral radius of $R_r$ is strictly
less than $\lambda_r$.  Thus defining $\Pi_r g_0 = c_r g_r$ for some $c_r >0$, we have
\begin{equation}
\label{eq:error split}
\lambda_r^{-k_r-1} \hLp_r^{k_r+1} (g_0-g_r) = (c_r - 1)g_r + \lambda_r^{-k_r-1}R_r^{k_r+1} g_0,
\end{equation}
where we have used the facts, $\Pi_r^2 = \Pi_r$, $\Pi_r g_r = g_r$ and $R_r g_r = 0$.  Integrating, we have
\[
\log \Big( 1 + \int \lambda_r^{-k_r-1} \hLp_r^{k_r+1} (g_0 - g_r) \, dm_0 \Big)
= \log \Big( c_r + \int \lambda_r^{-k_r-1}R_r^{k_r+1} g_0 \, dm_0 \Big)  .
\]

Now since the operators 
$\hLp_r$ have a uniform spectral gap for $r$ close to 0, there exists $\beta >0$ such that the spectral radius of $\lambda_r^{-1}R_r$ in $\B$
is less than $e^{-\beta}$ for all $r$ sufficiently small.  Since the variation norm dominates the $L^\infty$ norm, we have,
\[
\left| \int  \lambda_r^{-k_r-1} R_r^{k_r+1} g_0 \, dm_0 \right| \le \|  \lambda_r^{-k_r-1} R_r^{k_r+1} g_0 \| \le C e^{-\beta (k_r+1)}
\le C e^{-\beta s \mu_0(U_r)^{-\alpha} },
\]
for some fixed $C>0$, and this quantity is super-exponentially small in $\mu_0(U_r)$.  Moreover, since by \cite[Corollary 1]{KL pert}, the spectral projectors
$\Pi_r$ of $\hLp_r$ vary by at most $-\mu_0(U_r) \log \mu_0(U_r)$ for small $r$ and $\Pi_0 g_0 = g_0$, i.e., $c_0 =1$,  we have $|1-c_r| \le -C \mu_0(U_r) \log \mu_0(U_r)$, for some uniform $C>0$.

Using these estimates in the second term of \eqref{eq:alpha split} and dividing by
$s\mu_0(U_r)^{1-\alpha}$, the relevant expression becomes,
\[
\lim_{r \to 0} \frac{1}{s \mu_0(U_r)^{1-\alpha}} \log \Big( 1 + \mathcal{O}(-\mu_0(U_r) \log \mu_0(U_r)) \Big) .
\]
For $\alpha \ge 1$, it suffices to note that $\log (1 + \mathcal{O}(-\mu_0(U_r) \log \mu_0(U_r)) )$ converges to $0$ as $r \to 0$ to conclude that the above limit
vanishes.  For $\alpha \in (0,1)$, we note that in addition $\frac{\mu_0(U_r)\log \mu_0(U_r)}{\mu_0(U_r)^{1-\alpha}} \to 0$ as $r \to 0$,
which completes the proof of Theorem~\ref{thm:exp}. 


\subsection{Examples}
\label{sec:examples}

In this section, we provide examples of several classes of maps and potentials
for which our assumptions (F1)-(F4) of Section~\ref{sec:exp mix} hold.
More general examples, including the existence of a conformal measure for contracting
potentials, can be constructed using \cite{LSV1}.

\subsubsection{Lasota-Yorke maps of the interval with $\vf = -\log |Df|$.}

Such maps are assumed to admit a finite partition $\Z$ of $I$ into intervals on which
$f$ is differentiable and $|Df| \ge \sigma^{-1} > 1$.  $f$ is assumed to
be $C^2$ on the closure of each $Z \in \Z$.

The conformal measure $m$ is Lebesgue measure on $I$, and (F1)-(F3) are standard 
consequences of uniform expansion, the existence of $D^2f$ and the finiteness of the partition $\Z$.

Since the potential is bounded, condition (F4) can be guaranteed by the equivalent condition that for each
interval $J$, there exists $n(J)$, such that $f^{n(J)}(J) = I \bmod 0$.

Once we fix $z = \cap_{r>0} U_r$ and $n_1$ from \eqref{eq:n1 def}, {\bf (U)}
is always satisfied for $r$ sufficiently small due to the finiteness of $\Z^{n_1}$.
Thus Theorem~\ref{thm:exp} holds for this class of maps.

\subsubsection{Mixing Gibbs-Markov maps with large images.}
\label{gibbs}

Assume that $f(Z)$ is a union of elements of $\Z$ for each $Z \in \Z$,
where $\Z$ is the countable partition defined at the beginning of Section~\ref{sec:exp mix}.
Thus $\Z$ is a Markov partition for $f$.  

We assume that $f$ satisfies the big images and pre-images\footnote{This is automatic
if $f$ is full-branched.} (BIP) property: 
there exists a finite set
$\{ Z_j \}_{j \in \mathcal{J}} \subset \Z$ such that 
$\forall Z \in \Z$, $\exists j, k \in \mathcal{J}$ such that
$f(Z_j) \supseteq Z$ and $f(Z) \supseteq Z_k$.
We also assume that $|Df| \ge \sigma^{-1} > 1$ on each $Z \in \Z$.

We assume that $\vf$ is a potential which is Lipschitz continuous on each $Z \in \Z$, 
and admits a non-atomic 
conformal probability measure $m_\vf$ with $m_\vf(I \setminus \cup_{Z \in \Z} Z) =0$.

Then (F1) follows immediately from the regularity of $\vf$ and the expansion of $f$, and we have a Gibbs-Markov map.
Condition (F2) follows from the
existence of $m_\vf$ and (BIP) since by (F1), $\sup_Z e^{\vf} \le (1+C_d) m_\vf(Z)/m_\vf(f(Z))$:
\[
\sum_{Z \in \Z} \sup_Z e^{\vf} \leq
(1+C_d)\sum_{Z \in \Z} \tfrac{m_\vf(Z)}{m_\vf(f(Z))} \le (1+C_d) c_0^{-1},
\]
where $c_0 = \inf_{Z \in \mathcal{Z}} m_\vf(f(Z)) > 0$ by (BIP).

(F4) follows from mixing plus (BIP).  For (F3), we use the fact that for maps
satisfying our assumptions, the transfer operator $\Lp_\vf$
acting on functions which are Lipschitz on each element of $\Z$ is
known to have a spectral gap.  Since $\Lp_\vf^n 1$ converges to an invariant
density that is bounded away from $0$ (by (F4)), 
the expression on the right side of (F3) is bounded away from $0$ for all $n$ large
enough.  On the other hand, the expression on the left side
of (F3) must tend to $0$ by conformality and (F1), since for $Z \in \mathcal{Z}^n$,
\[
\sup_Z e^{S_n\vf} \le (1+C_d) \tfrac{m_\vf(Z)}{m_\vf(f^nZ)} \le (1+C_d) \frac{m(Z)}{c_0},
\]
and the diameter of $\Z^n$ must tend to $0$ by the expansivity of $f$.

Having verified (F1)-(F4), we may apply Theorem~\ref{thm:exp} to this class
of Gibbs-Markov maps.  Note that we can always arrange for
{\bf (U)} to be satisfied as long as we do not
choose $\cap_r U_r$ to be an accumulation point of the endpoints of the intervals in $\Z^{n_1}$.

\subsubsection{Gauss map, $f(x) = 1/x \bmod 1$.}
\label{gauss}

In this case, $\vf = - \log |Df|$, $m_\vf$ is Lebesgue measure on $[0,1]$, the
invariant density is $g_0 = \frac{1}{\ln 2}\frac{1}{1+x}$, and  $f$ is continuously
differentiable on each element of the partition $\Z = \{ Z_j \}_{j=1}^\infty$, $Z_j = (1/(j+1), 1/j)$.

For this potential, (F1) fails. However this system is well known to satisfy Rychlik's conditions 
since the potential is monotonic on each branch;
moreover the potential satisfies the weaker (H\"older) distortion control given
by the following lemma.

\begin{lemma}
\label{lem:distortion}
$\exists \, C_d > 0$ s.t. $|e^{S_n\vf(x) - S_n\vf(y)} -1| \le C_d |f^nx - f^ny|^{1/2}$,
whenever $f^ix$, $f^iy$ lie in the same element of $\Z$ for $i=0, 1, \ldots, n-1$.
\end{lemma} 

Before proving the lemma, we will verify the other conditions and show that in this case,
Lemma~\ref{lem:distortion} suffices to prove Proposition~\ref{prop:LY}, so that the
conclusions of Theorem~\ref{thm:exp} hold.

(F2) is immediate since $\sup_{Z_j} e^{\vf} \le Cj^{-2}$, $\forall j \ge 1$.

Notice that $|e^{\vf}|_\infty \le 1$, while $|e^{S_2\vf}|_\infty < 1$.  Thus the expression on the
left side of (F3) decreases exponentially in $n$, while $\Lp_\vf^n 1$ converges to $g_0$, which
is bounded away from $0$ on $I$.  Thus (F3) holds.

(F4) holds since $f$ is full-branched, and the potential satisfies the distortion control given by
Lemma~\ref{lem:distortion}.

We verify also that the items of Lemma~\ref{lem:pot} hold:  (a) holds by induction on (F2);
(b) holds with $C_d=1$ since $e^\vf$ is monotonic on each $Z_j$, so that 
$\bigvee_{Z_j} e^{\vf} \le \sup_{Z_j} e^{\vf}$; (c) holds by induction on (b), using (a).

Next we show that the operators $\hLp_r$ satisfy the uniform Lasota-Yorke inequalities
of Proposition~\ref{prop:LY} under assumption {\bf (U)}.  The assumption (F1) is used in precisely
two places in the proof of the proposition:  in equations \eqref{eq:dist} and \eqref{eq:var split}.
For \eqref{eq:dist}, the H\"older distortion control given by Lemma~\ref{lem:distortion} suffices
to give precisely the same bound.  For \eqref{eq:var split}, we use
$\bigvee_{J_i} e^{S_n\vf} \le \sup_{J_i} e^{S_n\vf}$ by the monotonicity of $e^{S_n\vf}$ on each $J_i$.

With these estimates, the contracting term in \eqref{eq:merge} becomes
$4 |e^{S_n \vf}|_\infty \bigvee_I \psi$ (it is the same expression, but with $C_d=1$).  
Thus we need only choose $n_1$ such that
$4 |e^{S_{n_1}\vf}|_\infty < 1$, replacing \eqref{eq:n1 def}, in order to prove the required
Lasota-Yorke inequalities under assumption {\bf (U)}.  Note that since
$f$ is full-branched, we can arrange for {\bf (U)} to be satisfied as long as $\{ z \} = \cap_r U_r$ 
is not chosen to be an endpoint of $\Z^{n_1}$.  

Turning to the proof of Theorem~\ref{thm:exp}, condition (F1) is used directly in one 
additional place: the proof of Lemma~\ref{lem:per}.  In that case, 
using the H\"older bound given by Lemma~\ref{lem:distortion}, we need only replace 
diam$(U_r)$ by $\sqrt{\mbox{diam}(U_r)}$ and choose $U_r$ sufficiently small that
$\sqrt{\mbox{diam}(U_r)} < \ve$.  Then the rest of the proof of Lemma~\ref{lem:close}
goes through without changes.

With these minor changes to the proof, the conclusions of Theorem~\ref{thm:exp} apply to the
Gauss map.

\begin{proof}[Proof of Lemma~\ref{lem:distortion}]
Let $x, y$ be as in the statement of the lemma, and let $f^ix \in Z_{j_i}$.  
The following bounds are elementary, yet essential to what follows,
 \[
 \sup_{Z_j} \frac{|Df^2|}{|Df|} \le C j \qquad \mbox{while} \qquad
 \mbox{diam}(Z_j) \le Cj^{-2} .
\]
Using these estimates, one may complete the standard (H\"older) distortion estimate,
\[
\begin{split}
\left| \log \frac{Df^n(x)}{Df^n(y)}\right| & \le \sum_{i=0}^{n-1} \left| \log |Df(f^ix)| - \log |Df(f^iy)| \right| \\
& \le \sum_{i=0}^{n-1} \sup_{Z_{j_i}}\frac{|Df^2|}{|Df|} |f^i x - f^iy| 
\; \le \; \sum_{i=0}^{n-1} C |f^ix - f^iy|^{1/2} \\
& \le C |f^n x - f^n y|^{1/2} \sum_{i=0}^{n-1} |e^{S_{n-i}\vf}|_\infty^{1/2}.
\end{split}
\]
The final sum converges exponentially in $i$ because
$|e^\vf|_\infty \le 1$ and $|e^{S_2\vf}|_{\infty} < 1$.
\end{proof}


\section{Results via inducing}\label{sec:inducing}
In this section, we consider some cases in which the map $f:I \to I$ and potential
$\vf$ do not satisfy
(F1)-(F4) of Section~\ref{sec:exp mix}.  In such cases, a common strategy is to 
consider an induced map to a subset of $I$ with stronger statistical properties.
This is the situation we shall address in this section: explicit examples will be given in the following section.

We begin with a map $f:I \to I$, a conformal measure $m_\vf$ with potential $\vf$,
and an invariant probability measure $\mu_\vf$, absolutely continuous with respect 
to $m_\vf$. We will fix the potential and simply denote this measure 
by $\mu$ in this section.

Fixing a sequence of sets $( U_r )_{r \in [0,r_0]}$, we assume
that we can select an interval $Y$ with $\mu(Y) >0$ and $U_r \subset Y \subset I$,
such that the first return map $F = f^{R_Y}:Y \to Y$
and the induced potential $\Phi = \sum_{i=0}^{R_Y -1} \vf \circ f^i$
satisfy (F1)-(F4).  

Let $\mu_Y := \frac{1}{\mu(Y)} \mu|_Y$
be the $F$-invariant probability measure,
and $\tau_{Y,r}(y) = \min\{ u \ge 1 : F^u(y) \in U_r\}$
be the first hitting time for the set $U_r$, which we sometimes refer to as the hole.
Let $R_{Y, u}(y) = \sum_{i=0}^{u-1} R_Y \circ F^i(y)$ be the $u^{\mbox{\scriptsize th}}$ 
return time to $Y$.

\begin{remark} 
We will assume for simplicity that the hole is always in $Y$.
This is not much of a restriction because it
is generically possible, once the location of the hole is known (and it is not at an 
indifferent fixed point or a recurrent critical point), 
to select a set $Y$ with good return map containing the hole.
\end{remark}

For $\mu_Y$-a.e.\ $y \in Y$, we have $R_{Y, u}/u \to 1/\mu(Y)$,
but for our purposes we need specific estimates for
the large deviations $\mu_Y(A_u)$ for the set
$$
A_u=A_{Y, u, \eps} := \{ y \in Y : \exists n \ge u\text{ such that } |R_{Y,n} - n/\mu(Y)| > n\eps\}.
$$

Following \eqref{eq:Icont}, we define $\Ycont$ to be the set of
points in $Y$ at which $F^k$ is continuous for all $k \in \N$.  

\begin{theorem}
\label{thm:induce}
Suppose $f:I\to I$ is as above and there exists $Y\subset I$ with $z \in \Ycont$ such that the first 
return map $F=f^{R_Y}:Y\to Y$ satisfies the assumptions of Theorem~\ref{thm:exp}.  

\begin{enumerate}
\item  If for any small $\eps>0$, there exists $c(\eps)>0$ such that $\mu_Y(A_u)\le e^{-c(\eps)u}$ for all large $u$, then for each $\alpha\in (0, \infty]$,
\begin{equation}
L_{\alpha, s}(z) = 
\begin{cases}
1, & \mbox{if $z$ is not periodic,} \\
1 - e^{S_p\vf(z)}, & \mbox{if $z$ is $p$-periodic for $f$}.
\end{cases}\label{eq:master}
\end{equation}

\item If there exist $\gamma\in (0,1)$ such that for any small $\eps>0$, there exist $C, c(\eps)>0$ such that $\mu_Y(A_u)\le C e^{-c(\eps)u^\gamma}$ for all large $u$, then \eqref{eq:master} holds for each 
$\alpha< \frac1{1-\gamma}$. 

\item If there exists $\gamma\in (0,1)$ and $C,c>0$ such that $\mu_Y(R_Y \ge u)\ge C e^{-c u^\gamma}$ for all large $u$,  then
 $L_{\alpha, s}(z)=0$ for $\alpha> \frac1{1-\gamma}$ and each $z \in I$. 

\item  If  both $\mu_Y(A_u)$ and $\mu_Y(R_Y \ge u)$ decay superpolynomially in $u$, but more slowly than any stretched exponential, then 
\eqref{eq:master} holds if $\alpha\le 1$ and $ L_{\alpha, s}(z) = 0$ if $\alpha >1$ for each $z\in \Icont$. 
\end{enumerate}
\end{theorem}

\begin{remark}
One expects, as in the examples of Section~\ref{sec:app}, that the decay of $\mu_Y(R_Y\ge u)$ matches that of $\mu_Y(A_u)$, so (2) and (3) in this theorem can be seen as complementary cases.
\end{remark}

\begin{remark}
Theorem~\ref{thm:induce} excludes the case $\alpha=0$ since as already noted in
Remark~\ref{rem:alpha 0}, the limit holds in this case under general conditions which do not
require a spectral gap.  Thus it is not necessary to pass to the induced map $F$ in this
case; one simply needs to verify the conditions listed in Remark~\ref{rem:alpha 0}.
\end{remark}


\subsection{Proof of Theorem~\ref{thm:induce}}
\label{induce proof}

We first prove the theorem for the quantity
$$
L_{Y,\alpha, s}(z) :=\lim_{r \to 0}  \frac{-1}{s \mu(U_r)^{1-\alpha}} \log \mu_Y(\tau_r > s\mu(U_r)^{-\alpha}).
$$

\begin{proposition}
Under the conditions of Theorem~\ref{thm:induce}, all parts of the theorem hold with $L_{Y,\alpha, s}(z)$ replacing $L_{\alpha, s}(z)$.  In particular, in cases (1) and (2),
\begin{equation}
L_{Y, \alpha, s}(z) = 
\begin{cases}
1, & \mbox{if $z$ is not periodic,} \\
1 - e^{S_p\vf(z)}, & \mbox{if $z$ is $p$-periodic for $f$}.
\end{cases}\label{eq:masterY}
\end{equation}
\label{prop:ind}
\end{proposition}

\begin{proof}
We will assume throughout that $\alpha\neq 1$ since this case is straightforward and proved elsewhere.
Fix some small $\eps > 0$, and assume that the hole $U_r$ 
is contained
inside one domain of $F$.  For notational simplicity, here we will assume that the centre $z$ of our $U_r$ is non-periodic, 
but the periodic case is then immediate. We remark only that if $z$ is periodic for $f$ with period $p$ and in the domain
of $F$, then $z$ is periodic for $F$ with period $q \le p$, and
$e^{S_q\Phi(z)} = e^{S_p\vf(z)}$.

If $y \in A_u^c$ and $t = u/\mu(Y)$, then  $\tau_r(y) > t$ implies that
$\tau_{Y,r}(y) > u/(1+\eps \mu(Y))$ and is implied by
$\tau_{Y,r}(y) > u/(1-\eps \mu(Y))$.
Since Theorem~\ref{thm:exp} applies to $(F, Y, \mu_Y)$, 
there exist values $\theta^+(v,r), \theta^-(v,r)$  
so that 
$$
\theta^-(v, r)e^{-v\mu_Y(U_r)^{1-\alpha}}\le \mu_Y(\tau_{Y, r}>v\mu_Y(U_r)^{-\alpha})\le \theta^+(v, r)e^{-v\mu_Y(U_r)^{1-\alpha}},
$$
where $\lim_{r\to 0}\frac{\log \theta^{\pm}(v, r)}{v\mu(U_r)^{1-\alpha}}=0$.

We compute for the path $t = s\mu(U_r)^{-\alpha}$, so
$$
u = s \mu(Y) \mu(U_r)^{-\alpha}= s \mu(Y)^{1-\alpha} \mu_Y(U_r)^{-\alpha}.
$$
We write
$\theta^-(r)=\theta^-\left(s(1-\eps\mu(Y))^{-1} \mu(Y)^{1-\alpha}, r\right)$ and 
$\theta^+(r) = \theta^+\left(s(1+\eps\mu(Y))^{-1} \mu(Y)^{1-\alpha}, r\right)$
to shorten notation.
Also we abbreviate
$$
G^\pm = \{ y \in Y :
\tau_{Y,r}(y) >  s(1\pm \eps \mu(Y))^{-1} \mu(Y)^{1-\alpha} \mu_Y(U_r)^{-\alpha}\}.
$$

First we bound $\mu_Y(\tau_r > t )$ from above, 
since $\{ \tau_r>t\wedge A_u^c\}\subset G^+$,
\begin{equation}
\label{eq:up tau}
\begin{split}
\mu_Y(\tau_r > t) &=  
 \mu_Y(\tau_r > t \wedge A_u^c) + \mu_Y(\tau_r > t \wedge A_u) \\
 &\le \mu_Y(G^+)+ \mu_Y(A_u)\\
 &\le \theta^+(r) e^{-s(1+\eps \mu(Y))^{-1}\mu(Y)^{1-\alpha}\mu_Y(U_r)^{1-\alpha}}
 +  \mu_Y(A_u)\\
 &=
  \theta^+(r) e^{-s(1+\eps \mu(Y))^{-1}\mu(U_r)^{1-\alpha}}
 +  \mu_Y(A_u).
\end{split}
\end{equation}

Similarly, we bound $\mu_Y(\tau_r > t)$ from below, using $G^-$:
\begin{align*}
\mu_Y(\tau_r > t \wedge A_u^c) 
&\ge\mu_Y(\tau_{Y,r} > u(1-\eps\mu(Y))^{-1} \wedge A_u^c) \\
&\ge  \mu_Y(\tau_{Y,r} >  s(1-\eps\mu(Y))^{-1} \mu(Y)\mu(U_r)^{-\alpha}) -\mu_Y(G^- \wedge A_u) \\
&\ge \theta^-(r)e^{- s(1-\eps\mu(Y))^{-1} \mu(U_r)^{1-\alpha}}-\mu_Y(G^- \wedge A_u),
\end{align*}
and therefore,
\begin{equation}
\label{eq:low tau}
\begin{split}
\mu_Y(\tau_r > t)  &= \mu_Y(\tau_r > t \wedge A_u^c) + \mu_Y(\tau_r > t \wedge A_u) \\
&\ge  \theta^-(r) e^{-s(1-\eps \mu(Y))^{-1} \mu(U_r)^{1-\alpha}}
+\mu_Y(\tau_r > t \wedge A_u) - \mu_Y(G^- \wedge A_u)\\
&\ge \theta^-(r) e^{-s(1-\eps \mu(Y))^{-1} \mu(U_r)^{1-\alpha}} - \mu_Y(A_u) .
\end{split}
\end{equation}

Now to find the limit in \eqref{eq:masterY}, we use first \eqref{eq:up tau}
to bound $L_{Y, \alpha, s}$ from below (taking a minus sign, so the inequality flips):

\begin{equation}
\begin{split}&\frac{-\log \mu_Y(\tau_r > t)}{s\mu(U_r)^{1-\alpha}}   \ge -\frac{\log\left(\theta^+(r) e^{-s(1+\eps \mu(Y))^{-1} \mu(U_r)^{1-\alpha}} + \mu_Y(A_u)\right)}{s\mu(U_r)^{1-\alpha}}\\
&\hspace{1cm}= -\frac{\log \theta^+(r)}{s\mu(U_r)^{1-\alpha}}+ \frac{1}{1+\eps \mu(Y)} -\frac{\log\left(1+\frac{e^{s(1+\eps \mu(Y))^{-1}\mu(U_r)^{1-\alpha}}}{\theta^+(r)}\mu_Y(A_u)\right)}{s\mu(U_r)^{1-\alpha}}.\end{split}\label{eq:master est}
\end{equation}

The first term converges to zero as $r\to0$ by assumption, so we focus on the final term.

\textbf{Case I: $\mathbf{\alpha\in (0,1)}$.}  
In this case since $\theta^+(r) = O(e^{s\mu(U_r)^{1-\alpha}+\delta})$ for any $\delta>0$, and $\mu(U_r)^{1-\alpha}\to 0$ as $r\to 0$, hence $\frac{e^{-s(1+\eps \mu(Y))^{-1}\mu(U_r)^{1-\alpha}}}{\theta^+(r)}=O(1)$ and we see that the final term of \eqref{eq:master est} is of order $ \frac{\mu_Y(A_u)}{\mu(U_r)^{1-\alpha}}$.  Assuming that $\mu_Y(A_u)\le Cu^{-\beta}$  for some $C, \beta>0$, 
we have,
$$\mu_Y(A_u)\le C(s \mu(Y) \mu(U_r)^{-\alpha})^{-\beta},$$ 
so that $L_{Y, \alpha,s}(z) \ge 1/(1 + \ve \mu(Y))$ if $\alpha-1+\alpha\beta>0$, i.e., 
$\alpha>\frac1{1+\beta}$.  So the lower bound for the non-degenerate part (i.e., the ``$\alpha \leq 1$'' part)
of (4) follows along with (2) and (1) for the $\alpha\in (0,1)$ case since $\eps$ was arbitrary.  
The upper bound follows immediately from Remark~\ref{rem:alpha}.

\textbf{Case II:  $\mathbf{\alpha\in (1,\infty)}$.}  Here we focus on the stretched exponential case since all remaining parts of this proposition then follow.  To complete the proof of (2), we again refer to \eqref{eq:master est}.  Suppose that there exist $C,c(\ve)>0$ and $\gamma\in (0,1)$ such that $\mu_Y(A_u)\le Ce^{-c(\ve)u^\gamma}$.  Then for \eqref{eq:masterY} to hold it is sufficient that the decay of 
$\mu_Y(A_u)$, which is $Ce^{-c(s\mu(U_r)^{-\alpha})^\gamma}$, is faster than 
$e^{-s(1+\eps \mu(Y))^{-1}\mu(U_r)^{1-\alpha}}$.  
So we require that $\alpha<\frac1{1-\gamma}$.  
The upper bound follows similarly, using \eqref{eq:low tau} in place of \eqref{eq:up tau}, completing (2).

To prove (3) and the degenerate (i.e., ``$\alpha > 1$'') part of (4), we assume that there exist $C,c>0$, $\gamma\in (0,1)$ such that $\mu_Y(R_Y \ge t) \ge C e^{-ct^\gamma}$.  Fix $\alpha>\frac1{1-\gamma}$.  
Then using the fact that $\{ \tau_r > t \} \supset \{ R_Y \ge t \}$, we estimate
\[
\begin{split}
\frac{- \log \mu_Y(\tau_r > t)}{s \mu(U_r)^{1-\alpha}}
& \le \frac{-\log \mu_Y(R_Y \ge t)}{s \mu(U_r)^{1-\alpha}} 
\; \le \; \frac{-\log (Ce^{-ct^\gamma})}{s\mu(U_r)^{1-\alpha}} \\
& \le \frac{- \log C}{s \mu(U_r)^{1-\alpha}} + \frac{cs^\gamma \mu(U_r)^{-\alpha \gamma}}{s \mu(U_r)^{1-\alpha}}
\end{split}
\]
and both terms tend to $0$ with $r$ since $\alpha > 1/(1-\gamma)$.  Note that this 
estimate easily extends from the measure $\mu_Y$ to the measure $\mu$,
so that $L_{\alpha,s}(z) = 0$ for all $z \in I$.

\textbf{Case III: $\mathbf{\alpha = \infty}$.}  For this case, we compute first the limit
$t\to \infty$ and then $r \to 0$ in the expression given by \eqref{eq: general lim}.  

Fix $\ve>0$ and define $A_u$ as before.  In analogy to the previous two cases, set
\[
G^{\pm} = \{ y \in Y : \tau_{Y,r}(y) > u/(1 \pm \ve \mu(Y)) \},
\]
where $u = t \mu(Y)$.  Notice then that
$\{ \tau_r > t \wedge A_u^c \} \subset G^+$ as before.  Thus as in \eqref{eq:up tau},
\[
\mu_Y(\tau_r > t) \le \mu_Y(G^+) + \mu_Y(A_u) .
\]
Following \eqref{eq:low tau}, we obtain,
\[
\mu_Y(\tau_r > t) \ge \mu_Y(G^-) - \mu_Y(A_u) .
\]

To prove the exponential case (1), assume that
there exists $c(\ve)>0$ such that $\mu_Y(A_u) \le C e^{-c(\ve) u}$.
Since $F$ satisfies the assumptions of Theorem~\ref{thm:exp}, we only consider $r$
so small such that all associated transfer operators $\hLp_r$ have a uniform
spectral gap.  Let $\Lambda_r$ denote the leading eigenvalue of $\hLp_r$ and
choose $r_1$ so small $\Lambda_r^{(1-\ve \mu(Y))^{-1}} > e^{-c(\ve)}$ for all $r < r_1$.
By Corollary~\ref{cor:uniform gap}, 
\[
C^{-1} \Lambda_r^{u(1 \pm \ve \mu(Y))^{-1}} \le \mu_Y(G^{\pm}) \le C \Lambda_r^{u(1\pm \ve \mu(Y))^{-1}}, 
\]
for some
$C >0$, independent of $t$, but possibly depending on $r$.

Thus on the one hand we derive a lower bound,
\begin{align*}
\lim_{t\to \infty} -\frac 1t \log \mu_Y(\tau_r > t) 
& \ge \lim_{t\to \infty} -\frac 1t \log \left( C \Lambda_r^{ t \mu(Y)/ (1+\ve \mu(Y))} + C e^{-c(\ve) t \mu(Y)} \right) \\
& = \frac{ - \mu(Y) \log \Lambda_r}{1 + \ve \mu(Y)} . 
\end{align*}
On the other hand, the analogous upper bound holds,
\begin{align*}
\lim_{t\to \infty} -\frac 1t \log \mu_Y(\tau_r > t) 
& \le \lim_{t\to \infty} -\frac 1t \log \left( C^{-1} \Lambda_r^{t \mu(Y)/(1-\ve \mu(Y))} - C e^{-c(\ve) t \mu(Y)} \right) \\
& = \frac{- \mu(Y) \log \Lambda_r}{1 - \ve\mu(Y)} . 
\end{align*}
Since $\ve>0$ was arbitrary, this yields 
\begin{equation}
\label{eq:Y esc}
\lim_{t \to \infty} - \frac 1t \log \mu_Y(\tau_r > t) = \lim_{t \to \infty} - \frac 1t \log \mu(\tau_r > t \wedge Y)
= - \mu(Y) \log \Lambda_r .
\end{equation}
Now using \eqref{eq:exp formula} applied to the induced map $F$, we conclude
\[
\lim_{r \to 0}  \lim_{t\to \infty} -\frac{1}{t \mu(U_r)} \log \mu_Y(\tau_r > t) 
= \lim_{r \to 0} \frac{-\log \Lambda_r}{\mu_Y(U_r)}  = 1
\]
in the generic case, and $1 - e^{S_p\vf(z)}$ in the periodic case.

For the remaining items (2)-(4) of the proposition, it suffices to show that $L_{Y, \alpha,s}(z) = 0$
when $\alpha = \infty$ under the assumption that $\mu_Y(R_Y \ge t) \ge C e^{-ct^\gamma}$
for some $\gamma \in (0,1)$.  This is a trivial estimate since in this case the escape rate
is $0$, i.e.,
\[
0 \le \lim_{t \to \infty} - \frac 1t \log \mu_Y(\tau_r > t) \le \lim_{t \to \infty} - \frac 1t \log \mu_Y(R_Y \ge t)
\le \lim_{t \to \infty} ct^{\gamma -1} = 0.
\]
It follows immediately that $L_{Y, \alpha, s}(z) = 0$ for all $z \in Y$. 
\end{proof}

\begin{proof}[Proof of Theorem~\ref{thm:induce}]
We will apply Proposition~\ref{prop:ind} to convert the results for $L_{Y, \alpha, s}$ to $L_{\alpha, s}$.  
For this, we turn to an extended system implied by the existence
of the first return map $F$.  We will refer to this as a Rokhlin tower (our map $F$ defines
what is nearly a Young tower, see \cite{young poly}, except that we do not require that 
$F$ have a Markov structure).  Define
\[
\Delta = \{ (y, n) \in Y \times \N : n < R_Y(y) \} . 
\]
The $\ell$th level of the tower is $\Delta_\ell = \{ (y,n) \in \Delta : n = \ell \}$ and the dynamics
is defined by $f_\Delta(y, n) = (y, n+1)$ if $n < R_Y(y) -1$ and $f_\Delta(y, R_Y(y)-1) = (F(y), 0)$.
The first return map to the base of the tower $\Delta_0 = Y$ is again $F = f^{R_Y}$.

The assumptions on $F$ imply that $\mu_Y$ is an invariant probability measure on
$Y = \Delta_0$, which induces an $f_\Delta$-invariant probability measure $\mu_\Delta$ on
$\Delta$:  Define $\mu_\Delta|_{\Delta_\ell} = c (f_\Delta)_*^\ell \mu_Y|_{f_\Delta^{-\ell}(\Delta_\ell)}$,
where $c = \mu_Y(R_Y)^{-1}$ is the normalising constant.  Letting $\pi : \Delta \to I$ denote the
natural projection, $\pi(x, \ell) = f^\ell(x)$, 
we have $\pi \circ f_\Delta = f \circ \pi$ and $\pi_* \mu_\Delta = \mu$.
 
Letting $\Delta^{(n)} = \cup_{\ell=0}^n \Delta_\ell$ denote the $n$ first levels of the tower, 
we observe that this gives us a sequence of induced maps $F_n:\Delta^{(n)}\to \Delta^{(n)}$ 
each satisfying the conditions of 
Theorem~\ref{thm:exp}.  (In fact, using the assumption on $F$, the potential for $(F_n)^n$ is contracting.)
The projection $\pi(\Delta^{(n)})$ gives a sequence of sets exhausting the space: $\mu(\pi(\Delta^{(n)}))\to 1$ as $n\to \infty$.  
We will carry out the proof for $\pi(\Delta^{(n)})$ in place of $Y$, but calling it $Y$ again and suppressing the index $n$. 

{\bf Case I: $\mathbf{\alpha \in (0,1)}$.}
We will use the facts 
\begin{enumerate}
\item[(a)] $\mu(\tau_r\le t\wedge Y)\le \mu(\tau_r\le t)$
\item[(b)] given $\gamma\in \R$, for $x$ small, $\log(1+\gamma x) \sim \gamma\log(1+x)$
\end{enumerate}
Then for $\alpha \in (0,1)$, $t=s\mu(U_r)^{-\alpha}$,
\begin{align*}-\frac{\log\mu_Y(\tau_r>t)}{s\mu(U_r)^{1-\alpha}}&= -\frac{\log(1-\mu_Y(\tau_r\le t))}{s\mu(U_r)^{1-\alpha}}= -\frac{\log\left(1-\mu(Y)^{-1}\mu(\tau_r\le t\wedge Y)\right)}{s\mu(U_r)^{1-\alpha}}\\
&\le -\frac{\log\left(1-\mu(Y)^{-1}\mu(\tau_r\le t)\right)}{s\mu(U_r)^{1-\alpha}}\sim -\frac1{\mu(Y)}\frac{\log\left(1-\mu(\tau_r\le t)\right)}{s\mu(U_r)^{1-\alpha}}\\
&=  -\frac1{\mu(Y)}\frac{\log\mu(\tau_r>t)}{s\mu(U_r)^{1-\alpha}}.
\end{align*}
Here we used (a) in the `$\le$ step' and (b) in the `$\sim$ step'.
So choosing $Y = \pi(\Delta^{(n)})$ exhausting our phase space, we deduce
\begin{equation}
\liminf_{r\to 0}\frac{-\log\mu(\tau_r>t)}{s\mu(U_r)^{1-\alpha}}\ge 
\begin{cases}
1, & \mbox{if $z$ is not periodic,} \\
1 - e^{S_p\vf(z)}, & \mbox{if $z$ is $p$-periodic for $f$}.
\end{cases}\label{eq:master 2}
\end{equation}
For non-periodic $z$, Remark~\ref{rem:alpha} gives the upper bound as 1 too, so $L_{\alpha, s}(z)=1$. 

For the periodic case we adapt Remark~\ref{rem:alpha} and use a result of \cite{FreFreTod15}.  
First recall $V_r:=U_r\cap f^{-p}(U_r)$ from the proof of Lemma~\ref{lem:per} and let $V'_r:=U_r\sm f^{-p}(U_r)$. 
For all small $r$, this will be a topological annulus around $z$.  By conditions \textbf{(P)} and (F1) (see \eqref{eq:translate}), 
\begin{equation}
\lim_{r\to 0}\frac{\mu(V'_r)}{\mu(U_r)}=\lim_{r\to 0}\frac{m_\vf(V'_r)}{m_\vf(U_r)}= 1-e^{S_p\vf(z)}.\label{eq:ratio}
\end{equation}
 We set $\tau_r':=\inf\{n\ge 1:f^n(x)\in V'_r\}$.  Now \cite[Proposition 2.7]{FreFreTod15} (with $B=U_r$ and $A=V'_r$) implies that
$$\mu(\tau_r'>n)-\mu(\tau_r>n)\le \sum_{j=1}^{p}\mu(\tau_r'>n \wedge f^{-n+j}(V_r))\le p\mu(V_r)< p\mu(U_r)$$
for all large $n$.  So we now proceed as in Remark~\ref{rem:alpha}:
\begin{align*}
0 &\le \frac{-\log \mu(\tau_r > t)}{s \mu(U_r)^{1-\alpha}} <
 \frac{-\log \left(\mu(\tau_r' > t)-p\mu(U_r)\right)}{s \mu(U_r)^{1-\alpha}} \\
 &=
\frac{-\log\left(1-\mu(\tau_r' \le t)-p\mu(U_r)\right)}{s \mu(U_r)^{1-\alpha}}=
\frac{-\log\left(1-\mu\left(\cup_{j=0}^{t-1} f^{-j}(V'_r)\right)-p\mu(U_r)\right)}{s \mu(U_r)^{1-\alpha}}\\
&
\le \frac{-\log\left(1-t \mu(V'_r)-p\mu(U_r)\right)}{s \mu(U_r)^{1-\alpha}}
 = \frac{-\log\left( 1-s\mu(V'_r) \mu(U_r)^{-\alpha} -p\mu(U_r)\right)}{s \mu(U_r)^{1-\alpha}}.
 \end{align*}
So by \eqref{eq:ratio}, the upper bound above converges to $1-e^{S_p\vf(z)}$ 
 as $\mu(U_r) \to 0$, so we conclude that $L_{\alpha, s}(z)=1 - e^{S_p\vf(z)}$.

{\bf Case II: $\mathbf{\alpha \in (1, \infty)}$.}

For $\alpha>1$, we obtain the following upper bound:
\begin{align*}
-\frac{\log\mu_Y(\tau_r>t)}{s\mu(U_r)^{1-\alpha}}& = \frac{\log\mu(Y)}{s\mu(U_r)^{1-\alpha}}-\frac{\log\mu(\tau_r>t\wedge Y)}{s\mu(U_r)^{1-\alpha}}\\
&\sim -\frac{\log\mu(\tau_r>t\wedge Y)}{s\mu(U_r)^{1-\alpha}} \ge -\frac{\log\mu(\tau_r>t)}{s\mu(U_r)^{1-\alpha}}.
\end{align*}
So we conclude that $L_{\alpha, s}(z) \le L_{Y,\alpha, s}(z)$.  
Note the above shows $L_{Y,\alpha, s}(z)=0$ implies $L_{\alpha, s}(z)=0$
so that items (3) and (4) of the theorem hold for $\alpha >1$.

To prove items (1) and (2) of the theorem, we also need a lower bound on $L_{\alpha,s}(z)$.
For this, recall that the measure $\mu$ can be expressed in terms of $\mu_Y$ by,
\[
\mu(A) = \frac{1}{\int R_Y \, d\mu_Y} \sum_{k=0}^\infty \sum_{i = 0}^k \mu_Y(f^{-i}(A) \cap Y_k),
\]
where $Y_k = \{ R_Y = k \}$, and $A$ is any measurable set.  Applying this expression
to $A = \{ \tau_r > t \}$, we note that $f^{-i}(\tau_r > t) \cap Y_k = \{ \tau_r > t+i \} \cap Y_k$
since $U_r \subset Y$.  Then reversing order of summation, we obtain,
\begin{equation}
\label{eq:Y sum}
\mu(\tau_r > t) = \frac{1}{\int R_Y \, d\mu_Y} \sum_{i = 0}^\infty \sum_{k=i}^\infty \mu_Y(\tau_r > t+i \wedge Y_k) 
\le \mu(Y) \sum_{i=0}^\infty \mu_Y(\tau_r > t+i) .
\end{equation}

To proceed, we prove a slight extension of our estimates in Section~\ref{ssec:deriv}.
Let $\hLp_r$ denote the punctured transfer operator for $F$ with potential $\Phi = S_{R_Y} \vf$
and hole $U_r$ as defined in \eqref{eq:open transfer}.  By
assumption on $F$ and Corollary~\ref{cor:uniform gap}, $\hLp_r = \Lambda_r \Pi_r + R_r$ 
has a uniform spectral gap, i.e., there exists $\beta>0$ such that the spectral radius of
$\Lambda_r^{-1} R_r$ is less than $e^{-\beta}$ for all $r$ sufficiently small.

\begin{lemma}
\label{lem:extend}
For all $r >0$ sufficiently small and
any $n \in \N$ such that $e^{-\beta n} < \mu_Y(U_r) \log \mu_Y(U_r)$, we have
\[
\mu_Y(\tau_{Y,r} > n) = \Lambda_r^n[1 + \mathcal{O}(\mu_Y(U_r)\log \mu_Y(U_r))] .
\]
\end{lemma}
\begin{proof}
Noting that \eqref{eq:alpha split} is valid for all iterates of $F$, we write
\[
\mu(\tau_{Y,r} > n) = \Lambda_r^n \left[ 1 + \int_Y \Lambda_r^{-n} \hLp_r^n (g_0 - g_r) \, dm\right],
\]
where $g_0$ and $g_r$ are the normalised eigenfunctions for $\Lp_0$ and $\hLp_r$, 
respectively.  Following \eqref{eq:error split}, we note that the error term above can be split
into two terms, one bounded by $Ce^{-\beta n}$ and the other by 
$-C\mu_Y(U_r) \log \mu_Y(U_r)$.  By assumption on $n$, the error is of order $\mu_Y(U_r) \log \mu_Y(U_r)$.
\end{proof}

Now fix $\ve >0$ and define $A_u = A_{Y,u,\ve}$ as before.  Recall that if $\tau_r(y) > n$
and $y \in A_{n\mu(Y)}^c$, then $\tau_{Y,r} > n\mu(Y)/(1+\ve \mu(Y))$.  
We assume that there exist $C, c(\ve), \gamma>0$ such that 
$\mu_Y(A_u) \le C e^{-c(\ve) u^\gamma}$, and require that $\alpha < \frac{1}{1-\gamma}$. 

For the
sake of brevity, set $\vartheta = L_{Y,\alpha,s}(z)$, and by \eqref{eq:exp formula}, we
may choose $r$ sufficiently small so that $\Lambda_r \le e^{-(1-\ve)\mu_Y(U_r) \vartheta}$.
Setting $n= t+i$,
$\rho_r = \mu_Y(U_r)\log \mu_Y(U_r)$,
and using Lemma~\ref{lem:extend}, we estimate each term in \eqref{eq:Y sum} by
\[
\begin{split}
\mu_Y(\tau_r > t+i) & \le \mu_Y\big(\tau_r > t+i \wedge A_{(t+i)\mu(Y)}^c\big) + \mu_Y(A_{(t+i)\mu(Y)}) \\
& \le \mu_Y\big(\tau_{Y,r} > (t+i)\mu(Y)/(1+\ve\mu(Y))\big) +\mu_Y(A_{(t+i)\mu(Y)}) \\
& \le \Lambda_r^{(t+i)\mu(Y)/(1+\ve\mu(Y))}[1 + \mathcal{O}(\rho_r)] + \mu_Y(A_{(t+i)\mu(Y)}) \\
& \le e^{-(1-\ve)\mu(U_r) \vartheta (t+i) / (1+ \ve \mu(Y))} [1 + \mathcal{O}(\rho_r)] + 
Ce^{-c(\ve)(t+i)^\gamma \mu(Y)^\gamma} .
\end{split}
\]
To estimate \eqref{eq:Y sum}, we must sum both terms above over $i$.  Recalling that
$t = s\mu(U_r)^{-\alpha}$, the sum over the first term is bounded by,
\[
\begin{split}
\sum_{i \ge 0} e^{-(1-\ve)\mu(U_r) \vartheta (t+i) / (1+ \ve \mu(Y))} [1 + \mathcal{O}(\rho_r)]
& = \frac{ [1+ \mathcal{O}(\rho_r)] e^{-(1-\ve)\vartheta s \mu(U_r)^{1-\alpha}/(1+ \ve \mu(Y))}}
{1- e^{-(1-\ve) \vartheta \mu(U_r)/(1+\ve \mu(Y))}} \\
& \le \frac{ 2 e^{-(1-\ve)\vartheta s \mu(U_r)^{1-\alpha}/(1+ \ve \mu(Y))}}
{ (1- \ve) \vartheta \mu(U_r)},
\end{split}
\]
for $\ve$ and $r$ sufficiently small.
The sum over the second term is (recalling that $c=c(\eps)$),
\[
\sum_{i \ge 0} Ce^{-c(t+i)^\gamma \mu(Y)^\gamma}
 \le C \int_0^\infty  e^{- c (t+x)^\gamma \mu(Y)^\gamma} \, dx
= \frac{C}{ c^{1/\gamma} \mu(Y) \gamma} \int_{ct^\gamma \mu(Y)^\gamma}^\infty e^{-y} y^{\frac{1}{\gamma} -1} \, dy,
\]
where we have changed variables, $y = c(t+x)^\gamma \mu(Y)^\gamma$.  Setting
$n = \lceil \frac{1}{\gamma} -1 \rceil$, we have $y^{\frac{1}{\gamma} -1} \le y^n$, so making
this substitution and integrating by parts $n$ times, yields
\[
\sum_{i \ge 0} Ce^{-c(t+i)^\gamma \mu(Y)^\gamma}
 \le \frac{C e^{-ct^\gamma \mu(Y)^\gamma}}{ c^{1/\gamma} \mu(Y) \gamma}
\sum_{k=0}^n \frac{n!}{k!} (ct^\gamma \mu(Y)^\gamma)^{n-k} 
\le \frac{e C n!\, t \,e^{-ct^\gamma \mu(Y)^\gamma} }{\gamma c^{\frac{1}{\gamma} -n} \mu(Y)^{1-\gamma n}}  .
\]

Putting these estimates together with \eqref{eq:Y sum}, we have,
\[
\mu(\tau_r > t) \le \mu(Y) \left[\frac{ 2 e^{-(1-\ve)\vartheta s \mu(U_r)^{1-\alpha}/(1+ \ve \mu(Y))}}
{ (1- \ve) \vartheta \mu(U_r)} + C' t e^{-c (s \mu(Y))^\gamma \mu(U_r)^{-\alpha \gamma}}\right] ,
\]
so that 
\[
- \log \mu(\tau_r > t) \ge \log \left(\frac{(1-\ve)\vartheta \mu(U_r)}{2 \mu(Y)}\right)
+ \frac{(1-\ve)\vartheta s \mu(U_r)^{1-\alpha}}{1+\ve\mu(Y)}
- \log [1 + B_r],
\]
where 
\[
B_r = \frac{C' s\mu(U_r)^{1-\alpha} (1-\ve) \vartheta}{2} e^{-c(s\mu(Y))^\gamma \mu(U_r)^{-\alpha \gamma} + (1-\ve)\vartheta s\mu(U_r)^{1-\alpha}/(1+\ve \mu(Y))}  .
\]
Note that $B_r \to 0$ as $r \to 0$ since $\alpha < \frac{1}{1-\gamma}$.  Thus dividing by 
$s \mu(U_r)^{1-\alpha}$ and recalling that $\alpha >1$, we have,
\[
\lim_{r \to 0} \frac{-\log \mu(\tau_r>t)}{s\mu(U_r)^{1-\alpha}}
\ge \vartheta \frac{1-\ve}{1+\ve\mu(Y)},
\]
which is the required lower bound since $\ve >0$ was arbitrary.  Thus 
$L_{\alpha,s}(z) = L_{Y, \alpha,s}(z)$ and items (1) and (2) of the theorem are proved
for this case.

{\bf Case III: $\mathbf{\alpha = \infty}$.}
First we note that an upper bound similar to the one derived in Case II holds:
\begin{equation}
\label{eq:upper Y}
\begin{split}
\lim_{t \to \infty} - \frac 1t \log \mu_Y(\tau_r > t)
& = \lim_{t \to \infty} \frac 1t \log \mu(Y) - \frac 1t \log \mu(\tau_r > t \wedge Y) \\
& \ge \lim_{t \to \infty} - \frac 1t \log \mu(\tau_r > t) .
\end{split}
\end{equation}

To prove items (2)-(4) of the Theorem, we must show $L_{\infty,s}(z) = 0$ under the assumption
$\mu_Y(R_Y \ge t) \ge C e^{ct^\gamma}$ for some $\gamma \in (0,1)$.  This 
follows from Case III in the proof of Proposition~\ref{prop:ind} since then $L_{Y, \infty,s}(z) = 0$.
Due to the upper bound above, $L_{\infty, s}(z) = 0$ as well.

To prove item (1) of the theorem, fix $\ve>0$ and assume there exists $c(\ve)>0$ such that
$\mu_Y(A_u) \le C e^{-c(\ve) u}.$  As in Case II, 
we take $r_0$ so small that the transfer operators
$\hLp_r$ associated with the induced map $F$ have a uniform spectral gap for all $r \in [0,r_0]$
and denote
their leading eigenvalues by $\Lambda_r$.  Using \eqref{eq:exp formula}, we choose $r_1 < r_0$ 
so small that
$e^{-(1-\ve)\mu_Y(U_r)\vartheta} \ge \Lambda_r \ge e^{-c(\ve)(1+\ve\mu(Y))/2}$ for all $r < r_1$.  

By \eqref{eq:Y esc} in the proof of Proposition~\ref{prop:ind}(1), and \eqref{eq:upper Y}, we have
\[
 \lim_{t \to \infty} - \frac 1t \log \mu(\tau_r > t) \le - \mu(Y) \log \Lambda_r 
 \]
To prove the corresponding lower bound, we follow \eqref{eq:Y sum} and the estimates
in Case II of the present proof (with $\gamma=1$).  In particular, using Lemma~\ref{lem:extend},
\[
\begin{split}
\mu_Y(\tau_r > t+i) & \le \mu_Y\big( \tau_r > t+i \wedge A_{(t+i)\mu(Y)}^c \big) 
+ \mu_Y\big( A_{(t+i)\mu(Y)}\big) \\
& \le \Lambda_r^{(t+i)\mu(Y)/(1+\ve \mu(Y))}[1 + \mathcal{O}(\rho_r)] + C e^{-c(\ve)(t+i)\mu(Y)} .
\end{split}
\]
Summing over $i$, we obtain
\[
\mu(\tau_r>t) \le \mu(Y) \left[ \frac{2 \Lambda_r^{t \mu(Y)/(1+\ve \mu(Y)} }{(1-\ve)\vartheta \mu(U_r)}
+ C e^{-c \mu(Y) t} \right] .
\]
And finally, 
\[
- \log \mu(\tau_r > t) \ge \log \frac{(1-\ve)\vartheta \mu(U_r)}{2 \mu(Y)}
- \frac{t \mu(Y) \log \Lambda_r}{1 + \ve \mu(Y)} - \log [ 1 + B_r],
\]
where $B_r \le \frac{C (1-\ve)\vartheta \mu(U_r)}{2} e^{-c \mu(Y) t/2}$, by choice of $r_1$.
Now dividing by $t$ and taking $t \to \infty$ yields
\[
\lim_{t \to \infty} - \frac 1t \log \mu(\tau_r > t) \ge - \frac{\mu(Y) \log\Lambda_r}{1+\ve \mu(Y)}.
\]
Since $\ve>0$ was arbitrary, our upper and lower bounds match.  Thus using again
\eqref{eq:exp formula}, we have
\[
\lim_{r \to 0} \lim_{t \to \infty} \frac{- \log \mu(\tau_r > t)}{t \mu(U_r)}
= \lim_{r \to 0} \frac{- \mu(Y) \log \Lambda_r}{\mu(U_r)} 
= \lim_{r \to 0} \frac{- \log \Lambda_r}{\mu_Y(U_r)} = 1
\]
in the generic case, and $1-e^{S_p\vf(z)}$ in the periodic case.  Thus $L_{\infty, s}(z) = L_{Y, \infty, s}(z)$
as required.
\end{proof}


\subsection{Remarks about the Polynomial Case}

Theorem~\ref{thm:induce} gives optimal results when the induced system has deviations
that are superpolynomial and when the decay rate of $A_u$ matches that of $\{ R_Y \ge u \}$.  
However, it gives only partial results if the induced system
has only polynomial deviations, i.e., $\mu_Y(A_u) \approx u^{-\beta}$ and $\mu_Y(R_Y \ge u) \ge u^{-\beta -1}$.  
In particular,
the proofs of Proposition~\ref{prop:ind}(4) and Theorem~\ref{thm:induce}(4) 
yields in the generic case,
\[
L_{\alpha, s}(z) = L_{Y, \alpha, s}(z)= 0 \; \; \; \mbox{if $\alpha > 1, \;$ and} \;  
L_{\alpha, s}(z) = L_{Y, \alpha, s}(z) = 1 \; \; 
\mbox{if $\frac{1}{1+\beta} < \alpha \le 1$}.
\]
It might appear that by improving our upper and lower bounds in \eqref{eq:up tau}
and \eqref{eq:low tau}, we might extend our results to the case $\alpha \le 1/(1+\beta)$,
but a closer look reveals there is a 
real obstruction to using inducing arguments to evaluate the required limits in
the polynomial case.  In particular, there is a nontrivial dependence between
the sets $\{ \tau_r > t \}$ and $A_u$ which makes the polynomial case
particularly delicate from this point of view.

To illustrate this point, consider the class of Manneville-Pomeau or LSV maps
on the unit interval, defined by
\[
f(x) = \begin{cases}
x + 2^\gamma x^{1+\gamma}, & \mbox{for } x \in [0,1/2), \\
2x-1, & \mbox{for } x \in [1/2,1].
\end{cases}
\]
When $\gamma \in (0,1)$, these maps preserve an invariant
probability measure $\mu$, absolutely continuous with respect to Lebesgue,
with density $g \approx x^{-\gamma}$ for $x$ near $0$ \cite{young poly, LSV2}.

Set $Y = [1/2,1]$ and let $( U_r )_{r \in (0,r_0]} \subset (1/2+\delta, 1)$, for some $\delta >0$. 
For $k \ge 0$, let $a_k = f_L^{-k}(1/2)$, where $f_L$ is the left branch of $f$.
Set $J_0 = Y$ and $J_k =  [a_k, a_{k-1})$ for $k \ge 1$.
Note that $\tau_Y = k+1$ on $f_R^{-1}(J_k)$, where $f_R$ is the right branch of $f$.

We claim $\{ \tau_r > t \wedge A_u \} \supseteq \bigcup_{k \ge t} f_R^{-1}(J_k)$.
Note that $\bigcup_{k \ge t} f_R^{-1}(J_k) = \{ R_Y > t\}$, and that
$\{ \tau_r > t \} \supset \{ R_Y > t \}$ since $U_r \subset Y$.  Moreover,
if $R_Y(x) > t$, then $\tau_{Y,u}(x) > u-1 + t$, and for $u = \mu(Y) t$, we have
\[
\frac 1u \tau_{Y,u}(x) > 1 - \frac 1u + \frac{1}{\mu(Y)}
\; \; \implies \; \; 
\frac 1u \tau_{Y,u}(x) - \frac{1}{\mu(Y)}  > 1 - \frac 1u,
\] 
so that $x \in A_u$ for all $u \ge 2$ and $\ve < 1/2$.  Thus $\{ R_Y > t \} \subset A_u$,
completing the proof of the claim.

Using well-known
estimates \cite{LSV2} on the spacing of $a_k$,  $a_k \approx k^{-1/\gamma}$,
\begin{equation}
\label{eq:bad}
\mu_Y(\tau_r > t \wedge A_u) \ge c t^{-1/\gamma} = c s^{-1/\gamma}  \mu(U_r)^{\alpha/\gamma},
\end{equation}
for some uniform constant $c>0$, where as usual we have set $t = s\mu(U_r)^{-\alpha}$.  
Using this lower bound, we may split up the relevant
expression in the limit defining $L_{Y,\alpha,s}$ as follows, 
\begin{equation}
\label{eq:split}
\frac{-\log \mu_Y(\tau_r > t)}{s\mu(U_r)^{1-\alpha}}
= \frac{- \log \mu_Y(\tau_r > t \wedge A_u^c)}{s\mu(U_r)^{1-\alpha}}
-\frac{\log[1 + \frac{\mu_Y(\tau_r > t \wedge A_u)}{\mu_Y(\tau_r > t \wedge A_u^c)}]}{s\mu(U_r)^{1-\alpha}}.
\end{equation}
To use the results for the induced map, one would expect that the first term above tends to the
desired limit, while the second term above acts as an error term and tends to $0$ as $r \to 0$.
However, using \eqref{eq:bad}, we see that the `error' term is bounded below by
\[
\frac{\mu_Y(\tau_r > t \wedge A_u)}{s\mu(U_r)^{1-\alpha}}
\ge c' \mu(U_r)^{-1+\alpha + \alpha/\gamma} \xrightarrow[r \to 0]{} \infty,
\]
whenever $\alpha < \gamma/(1+\gamma)$.

By Remark~\ref{rem:alpha}, we known all limit points of $L_{Y,\alpha,s}$ lie in $[0,1]$,
so in the range $\alpha < \gamma/(1+\gamma)$, the limit relies on cancellation between 
two diverging terms in \eqref{eq:split}.  This implies that what we would like to consider
to be an error term does not function as one for small $\alpha$.


\section{Applications of inducing}\label{sec:app}
Theorem~\ref{thm:induce} applies whenever we have a system $(I, f, \mu)$ with an inducing scheme $(X, F, \nu)$ where $F=f^\tau$ and $\tau$ is the first return time to $X$ where, moreover, $\nu(A_u)$ is known to satisfy a suitable large deviations principle.  
At present such large deviations principles are known in quite specific cases.   We mention several examples here.

\subsection{Generalised Farey maps}
In the i.i.d.~case it has been shown that the large deviation rate of an unbounded observable $\psi$ matches the tail of the observable.  
For example for $\bar\psi=\int\psi~d\nu$, $\gamma\in (0,1)$ and $c>0$,
$$\nu(\psi>n)\le ce^{-n^\gamma} \Longrightarrow \lim_{n\to \infty}\frac1{n^\gamma}\log\nu(S_n\psi>\eps+\bar\psi)=-\eps^\gamma,$$
where $S_n\psi$ is the $n$-th ergodic sum of these observables, see \cite{GanRamRem14}.  Similarly, if the tail of an observable is polynomial of order $\beta$, then the deviations are polynomial of order $\beta-1$; and for exponential, the orders match exactly \cite{Gantert}.

An application of Theorem~\ref{thm:induce} is to generalised Farey maps as in \cite{KesMunStr12}.  Here one chooses a countable partition $\{A_n\}_n$ of $(0,1]$ by left-open, right-closed intervals labelled in decreasing order in the interval with length of $A_n$ equal to $a_n$ for each $n$.  Then for $t_n :=  \sum_{k=n}^\infty a_k$ and $x\in [0,1]$,
\begin{equation*}
f(x)=\begin{cases} (1-x)/a_1 & \text{ if } x\in A_1,\\
a_{n-1}(x-t_{n+1})/a_n+t_n & \text{ if } x\in A_n, n\ge 2,\\
0 & \text{ if } x=0.
\end{cases}
\end{equation*}
Lebesgue measure is invariant for this map and taking a first return map to the interval $A_1$ gives us an induced map satisfying the conditions of 
Theorem~\ref{prop:LY}. Since the branches are linear, the
map behaves in an i.i.d.~manner so that Lebesgue measure is a Markov measure for the 
induced map.

Moreover, one can choose the 
intervals $\{A_n\}_n$ in such a way that any of the tail decay conditions given by 
$(t_n)_n$ apply to our observable $R_Y$. By the results above these match 
the large deviations, so we may also apply the appropriate items of Theorem~\ref{thm:induce}.  
We observe that the only points $z$ which this theorem does not apply to directly are $\cup_{n\ge0} f^{-n}0$.  
It is straightforward to adapt the theorem slightly to cover all elements of this set except 0.


\subsection{Maps with exponential tails}

If we start with an interval map $f:I\to I$ and can find a well-behaved first return map 
to an interval $Y \subset I$ with exponential tails, then Theorem~\ref{thm:induce}(1) holds.  
That is, we require the first return  map $F = f^{R_Y}$ to be a full-branched Gibbs-Markov map 
where 
the induced measure $\mu_Y$ has $\mu_Y(n-1 \le R_Y < n) \le C e^{-\beta n}$
for some constants $C, \beta >0$.   By Section~\ref{gibbs}, $F$ satisfies (F1)-(F4) of 
Section~\ref{sec:exp mix}.

The fact that a full-branched Gibbs-Markov map has exponential large deviations for observables with 
exponential tails appears to be 
essentially folklore. Yuri \cite{Yur05} quotes such a result, 
but the setting is slightly different and the proof there is not given explicitly, so for completeness,
we provide the proof in Appendix~\ref{app b}.
Since $R_Y$ has exponential tails, it follows from Proposition~\ref{prop:ld} and
Corollary~\ref{cor:exp tail} that $R_Y$ 
satisfies a (local) exponential large deviations estimate and thus Theorem~\ref{thm:induce}(1) 
applies to the original system $f:I\to I$.

We remark that by this argument, Theorem~\ref{thm:induce}(1) applies to the tower map 
$(f_\Delta, \Delta)$ whenever one can construct
a Young tower \cite{young poly} over an interval as described in the proof of Theorem~\ref{thm:induce}.

In order to develop a specific class of examples, for the remainder of this section, we make
the following standing assumptions.  We assume
that $f: I \to I$ is a $C^2$, topologically mixing unimodal map with critical 
point $c$ and 
$\orb(c) = \{ f^n(c) : n \ge 1 \}$ nowhere 
dense.\footnote{We note that we can drop the topologically mixing and unimodal assumptions, but this makes our statements more involved. Similarly, one can also drop the requirement
that $\orb(c)$ be nowhere dense, see for example \cite{DemTodd}.}
Then one can find an interval $Y$, compactly contained in $I \setminus \overline{\orb(c)}$, 
such that $(Y, F)$ is full branched (see \cite[Chapter 4]{MSbook} for details), where $F$ is the
first return map to $Y$.  
Moreover we assume $F$ has bounded distortion, e.g.\ $f$ has negative Schwarzian derivative;
then $(Y, F)$ is Gibbs-Markov.  
Finally, we assume that our measure is an equilibrium state for some $\vf$ and discuss when our induced system has (F1)--(F4) 
and the return time has exponential tails so that we can conclude that Theorem~\ref{thm:induce}(1) holds.

\subsubsection{Collet-Eckmann case}
If $f$ satisfies the Collet-Eckmann condition (i.e., $|Df^n(c)|$ grows exponentially,
and this case includes Misiurewicz maps, i.e., $c$ is not recurrent
nor attracted to a stable periodic orbit, provided $f$ is non-flat at $c$), 
then for $\vf=-t\log|Df|$, there is a unique equilibrium state  for each $t$ in a neighbourhood of $[0,1]$, see for example \cite{PrzRiv11}.  Moreover, it can be deduced (e.g.\ from \cite{PrzRiv11}) that  $(Y, F)$ satisfies the conditions (F1)--(F4) for Theorem~\ref{thm:exp} to hold for the induced version of $\mu_t$  (note that the conformal measure is w.r.t.\ the normalised potential $\vf-P(\vf)$), and that the return time has exponential tails.  
By Corollary~\ref{cor:exp tail}, $R_Y$ enjoys exponential large deviations with 
respect to the equilibrium measure $\mu_t$.
Thus choosing $z \in Y_{cont}$ so that {\bf (U)} is satisfied, it follows that
Theorem~\ref{thm:induce}(1) holds for each $t$ in a neighborhood of $[0,1]$.

\subsubsection{Non-Collet Eckmann case}
If $f$ fails the Collet-Eckmann condition, then for the potential $\vf=-t\log|Df|$, there is still 
a unique equilibrium state  for $t\in (t_0, 1)$ for some $t_0<0$, again see for example \cite{PrzRiv11}.  
Moreover, $(Y, F)$ satisfies the conditions for Theorem~\ref{thm:exp} 
to hold for the induced version of $\mu_t$, and the return time has exponential tails.
So again choosing $z \in Y_{cont}$ so that {\bf (U)} is satisfied,
Theorem~\ref{thm:induce}(1) holds for this class of potentials.
By contrast, for $t=1$, even if there is an equilibrium state for $-\log|Df|$, it will have sub-exponential mixing, so Theorem~\ref{thm:induce}(1) will fail.

\subsubsection{Lipschitz potentials} 
If $\vf$ is a Lipschitz potential, then our results hold more generally: for Theorem~\ref{thm:induce}(1) to hold for the equilibrium state we only need the potential to be \emph{hyperbolic}, 
i.e., $\sup_{x\in I}\frac 1nS_n\vf(x)<P(\vf)$ for some $n$, 
where $P(\vf)$ denotes the variational pressure.  
As shown in \cite{LiRiv14} this is automatic if we merely assume that $|Df^n(f(c))|\to \infty$.

\appendix

\section{Proof of Proposition~\ref{prop:LY}}
\label{appendix}

The $L^1$ bound on $\hLp_r^n \psi$ in Proposition~\ref{prop:LY} follows directly
from \eqref{eq:cov}, so we focus on proving the required bound on the variation of $\hLp_r^n \psi$.

For $r \in [0,r_0]$,
let $\cI_r^n = \{ J_i \}_i = \{ (a_i, b_i) \}$ denote the intervals of monotonicity for $\f_r^n$ and set
$K_i = f^n(J_i)$.  
Then for $\psi \in \B$  and $n \ge 0$, we estimate,
\begin{equation}
\label{eq:first}
\begin{split}
\bigvee_I \hLp_r^n \psi & \le \sum_i \bigvee_{J_i} \psi e^{S_n\vf} + \psi(a_i) e^{S_n\vf(a_i)}
+ \psi(b_i)e^{S_n\vf(b_i)} \\
& \le \sum_i 2 \bigvee_{J_i} \psi e^{S_n\vf} + \frac{1}{m_0(J_i)} \int_{J_i} \psi e^{S_n\vf} \, dm_0
\end{split}
\end{equation}
For the second term in \eqref{eq:first}, we note that by conformality and 
the bounded distortion property (F1),
we have for each $x \in J_i$,
\begin{equation}
\label{eq:dist}
e^{S_n\vf(x)} \cdot \tfrac{m_0(K_i)}{m_0(J_i)} \le 1+C_d . 
\end{equation}
For the first term in \eqref{eq:first}, we split
\begin{equation}
\label{eq:var split}
\bigvee_{J_i} \psi e^{S_n\vf} \le \sup_{J_i} e^{S_n\vf} \bigvee_{J_i} \psi
+ \sup_{J_i} |\psi|  \bigvee_{J_i} e^{S_n\vf} 
\le \sup_{J_i} e^{S_n\vf} \bigvee_{J_i} \psi
+ \sup_{J_i} |\psi|  C_d \sup_{J_i} e^{S_n\vf},
\end{equation}
where we have used Lemma~\ref{lem:pot}(b) to bound $\bigvee_{J_i} e^{S_n\vf}$.  Using the bound
$\sup_{J_i} |\psi| \le \bigvee_{J_i} \psi + (m_0(J_i))^{-1} \int_{J_i} |\psi | \, dm_0$, we 
put these estimates together in \eqref{eq:first} and use \eqref{eq:dist} to obtain,
\begin{equation}
\label{eq:merge}
\begin{split}
\bigvee_I \hLp_r^n \psi 
& \le  \sum_i (2+2C_d) \sup_{J_i} e^{S_n\vf} \bigvee_{J_i} \psi
+  \frac{(1+C_d)(1+2C_d)}{m_0(K_i)} \int_{J_i} |\psi|  \, dm_0  \\
& \le (2+2C_d) |e^{S_n\vf}|_\infty \bigvee_I \psi + \inf_i \frac{(1+C_d)(1+2C_d)}{m_0(K_i)} \int_{\I^{n-1}_r} |\psi| \, dm_0 .
\end{split}
\end{equation}
Applying \eqref{eq:merge} when $n = n_1$,  setting $\bar{\sigma} := (2+2C_d) |e^{S_{n_1}\vf}|_\infty < 1$,
and using {\bf (U)} yields,
\[
\bigvee_I \hLp_r^{n_1} \psi \le \bar{\sigma} \bigvee_I \psi + \frac{(1+2C_d)^2}{c_0} \int_{\I^{n-1}_r} |\psi| \, dm_0,
\] 
and since $\bigvee_I e^{S_n\vf} < \infty$ for each $n$, this relation can be iterated
to complete the proof of Proposition~\ref{prop:LY} with $\sigma = \bar{\sigma}^{1/n_1}$.


\section{Exponential deviations}
\label{app b}

In this section, we prove the fact that full-branched Gibbs-Markov maps have
exponential large deviations for observables with exponential tails. 

Let $P_G(\phi)$ denote the Gurevich pressure of $\phi$ (see \cite{Sar99}).  Note that as in \cite[Theorem 2]{Sar99}, 
this is equal to the variational definition of pressure 
given in Section~\ref{sec:exp mix}.

\begin{proposition}\label{prop:ld}
Let $F$ be a full-branched Gibbs-Markov map and $\phi, \psi$ weakly H\"older continuous potentials.  If
 there exists $\delta > 0$ such that
$P_G(\phi + t \psi) < \infty$ for each $|t|< \delta$ (or equivalently that $|\Lp_{\phi+t\psi} 1|_\infty < \infty$), then $\psi$ enjoys exponential large deviations for $\mu$ the equilibrium state for $\phi$.
\end{proposition}

\begin{proof}
First we note that the assumptions on $\phi$ imply: (i) $\phi$ has finite Gurevich pressure
$P_G(\phi)$ \cite[Theorem 1]{Sar99}; (ii) $\phi$ is positive recurrent \cite[Corollary 2]{Sar03};
and (iii) there exists a finite conformal Borel measure $m_\phi$, positive on cylinders, such that
$\frac{dm}{dm \circ F} = e^{\phi - P_G(\phi)}$ \cite[Theorem 4, Proposition 3]{Sar99}.

Under these conditions, the associated transfer operator $\Lp_\phi$ acting on the space of
H\"older continuous functions\footnote{H\"older continuity here is defined using the
same constant $\theta$ as for the potential $\phi$, i.e., weak H\"older continuity of $\phi$ means
$\sup_{C_n^i \in \mathcal{P}_n}\sup \{|\phi(x) - \phi(y)| : x, y \in C_n^i \} \le \theta^n$, 
where $\mathcal{P}_n$ is 
the set of $n$-cylinders for $F$. We study the transfer operator on the class of functions
$f : X \to \mathbb{R}$ sharing the same property as $\phi$.} has a spectral gap.
It then follows from \cite[Theorem 2.1]{CyrSar}, that $\phi$ is strongly positive recurrent.
(We refer the reader to \cite{CyrSar} for the relevant definition.) 

Strong positive recurrence implies that if $\psi$ is a weakly H\"older continuous function 
such that $P_G(\phi + t\psi) < \infty$ for all $|t| < \delta$ and some $\delta > 0$, then
$t  \mapsto P_G(\phi + t\psi)$ is analytic in $t$ \cite[Theorem 3]{Sar01}.  
Moreover, $\phi + t\psi$ is positive recurrent for each $|t| < \delta$ and
has a Gibbs measure $\mu_t$ which is moreover the unique equilibrium state for $\phi+t\psi$.  
Denote by $\mu = \mu_0$ the Gibbs measure for $\phi$.
Without loss of generality, in what follows we assume $P_G(\phi) = \mu(\psi) = 0$ 
and that $0$ is a local minimum for $t\mapsto P_G(\phi+t\psi)$.

Now define $J_n^+(\eps)$ to be the collection of $n$-cylinders containing a point $x$ so that 
$S_n\psi(x)>n\eps$; similarly, let $J_n^-(\eps)$ denote the collection of $n$-cylinders
containing $x$ such that $S_n\psi(x) <  -n \eps$.  
We first consider $J_n^+(\eps)$.  Since $\frac{d}{dt} P_G(\phi+t\psi)\vert_{t=t_0}=\mu_{t_0}(\psi)$ for 
$|t_0| < \delta$, and by continuity of the derivative for $\eps$ small enough we can find 
$q>0$ so that $\mu_q(\psi)=\eps$. 

Then strict convexity of pressure implies that $P_G(q\psi+\phi)-q \eps<0$ (a slightly more sophisticated argument allows us to express this in terms of the Helmholtz free energy, but we do not require this here). 

Let $\P_n$ denote the set of $n$-cylinders and for $C_n^i\in \P_n$, let $x_n^i$ be 
the fixed point of $F^n$ in $C_n^i$.  So we compute, using the Gibbs property (here the constant $C$ covers the Gibbs constant and distortion constants):
\begin{align*}
\mu\left(S_n\psi>n\eps\right)&\le \sum_{C_n^i\in J_n^+(\eps)}\mu(C_n^i) \le  C\sum_{C_n^i\in J_n^+(\eps)}e^{S_n\phi(x_n^i)} \\
&\le  C^2\sum_{C_n^i\in J_n^+(\eps)}e^{q(S_n(\psi-\eps))(x_n^i))+S_n\phi(x_n^i)}\\
&\le  C^2 e^{-nq \eps}\sum_{C_n^i\in \P_n}e^{S_n(\phi+q\psi(x_n^i))}
\end{align*}
Taking logarithms, dividing by $n$ and taking limits we obtain
$$\limsup_{n\to\infty}\frac1n\log \mu\left(S_n\psi>n\eps \right)\le P(\phi+q\psi)-q \eps<0$$
as required.  A similar argument, with $q<0$, applies to $J_n^{-}(\eps)$.
\end{proof}

\begin{corollary}\label{cor:exp tail}
Under the assumptions of Proposition~\ref{prop:ld}, suppose that $\psi$ is weakly H\"older continuous 
with exponential tails, i.e.,
$\mu(n-1<|\psi|\le n)\le e^{-\beta n}$, for some $\beta > 0$. Then $\psi$ enjoys
exponential local large deviations with respect to $\mu$.
\end{corollary}

\begin{proof}
Letting $\{ x_j \}_j$ be the collection of all fixed points of $F$, and 
$\psi_i$ be the maximum value $|\psi|$ takes on the 1-cylinder $X_i$,  by the Gibbs property, 
\begin{align*}
\sum_{j}e^{(\phi+t\psi)(x_j)} &=\sum_{n\ge1}\sum_{n-1< \psi_j\le n}e^{(\phi+t\psi)(x_j)}\le C^2\sum_ne^{n|t|} \mu(n-1<|\psi|\le n)\\
&\le C^2\sum_ne^{n(|t|-\beta)} <\infty
\end{align*}
provided $|t|<\beta$.  Standard theory shows that this implies that 
$P_G(\phi+t\psi)<\infty$, so that $\psi$ satisfies the hypotheses of Proposition~\ref{prop:ld}.
\end{proof}


\begin{thebibliography}{99999}

\bibitem[AFV]{AytFreVai15} H.\ Ayta\c{c}, J.M.\ Freitas, S.\ Vaienti,
\emph{Laws of rare events for deterministic and random dynamical systems,} Trans. Amer. Math. Soc. \textbf{367} (2015), 8229--8278.

\bibitem[BT]{BruTod09} H.\ Bruin, M.\ Todd,
\emph{ Return time statistics of invariant measures for interval maps with positive Lyapunov exponent,}
 Stoch.\ Dyn.\ \textbf{9} (2009), 81--100.

\bibitem[BY]{BY}  L.~Bunimovich, A.~Yurchenko, \emph{Where to place a hole to achieve
a maximal escape rate}, Israel J.\ of Math.\ {\bf 182} (2011), 229--252.

\bibitem[CS]{CyrSar} V.~Cyr, O.~Sarig, \emph{Spectral gap and transience for Ruelle operators
on countable Markov shifts}, Commun. Math. Phys.{\bf 292} (2009), 637--666.

\bibitem[DF]{DemFer} M.F.~Demers, B.~Fernandez, \emph{Escape rates and singular limiting
distributions for intermittent maps with holes}, Trans.\ Amer.\ Math.\ Soc.\ {\bf 368} (2016), 4907--4932.

\bibitem[DT]{DemTodd}  M.F.~Demers, M.~Todd, \emph{Equilibrium states, pressure
and escape for multimodal maps with holes}, Israel J. Math. {\bf 221} (2017), 367--424.

\bibitem[FP]{FerPol12} A.\ Ferguson, M.\ Pollicott,
\emph{ Escape rates for {G}ibbs measures,}
 Ergodic Theory Dynam.\ Systems, \textbf{32} (2012), 961--988.

\bibitem[FFT1]{FreFreTod13} A.C.\ Freitas, J.M.\ Freitas, M.\ Todd,
\emph{The compound Poisson limit ruling periodic extreme behaviour of non-uniformly hyperbolic dynamics,} 
Comm. Math. Phys. \textbf{321} (2013), 483--527.  
 
\bibitem[FFT2]{FreFreTod15} A.C.\ Freitas, J.M.\ Freitas, M.\ Todd,
\emph{Speed of convergence for laws of rare events and escape rates,} 
Stochastic Process. Appl. 1\textbf{25} (2015), 1653--1687. 


\bibitem[G]{Gantert} N.\ Gantert,
{\em Large deviations for a heavy-tailed mixing sequence,}
Preprint TU Berlin.

\bibitem[GRR]{GanRamRem14} N.\ Gantert, K.\ Ramanan, F.\ Rembart,
\emph{ Large deviations for weighted sums of stretched exponential random variables,}
Electron. Commun. Probab. \textbf{19} (2014), 1--14.

\bibitem[H]{Hay13} N.\ Haydn,
\emph{ Entry and return times distribution,}
 Dyn. Syst. \textbf{ 28} (2013), 333--353.

\bibitem[KL1]{KL pert} G.\ Keller, C.\ Liverani, 
\emph{Stability of the spectrum for transfer operators,} 
Annali della Scuola Normale Superiore di Pisa, 
Classe di Scienze (4) Vol. XXVIII, (1999), 141--152.

\bibitem[KL2]{KL zero} G.\ Keller, C.\ Liverani, 
\emph{Rare events, escape rates and quasistationarity: some exact formulae,} 
Journal of Statistical Physics {\bf 135} (2009), 519--534.  

\bibitem[KMS]{KesMunStr12} M.\ Kesseb\"ohmer, S.\ Munday, B.O.\ Stratmann, 
\emph{Strong renewal theorems and Lyapunov spectra for $\alpha$-{F}arey and $\alpha$-{L}\"uroth systems,} Ergodic Theory Dynam.\ Systems \textbf{32} (2012), 989--1017. 

\bibitem[LR]{LiRiv14} H.\ Li, J.\ Rivera-Letelier,
\emph{ Equilibrium states of weakly hyperbolic one-dimensional maps for H\"older potentials,}
 Comm. Math. Phys. \textbf{328} (2014) 397--419.

\bibitem[LSV1]{LSV1} C.~Liverani, B.~Saussol, S.~Vaienti, \emph{Conformal measure
and decay of correlation for covering weighted systems}, Ergodic Theory Dynam.\ Systems
{\bf 18}:6 (1998), 1399--1420.

\bibitem[LSV2]{LSV2} C.~Liverani, B.~Saussol, S.~Vaienti, \emph{A probabilistic approach
to intermittency},  Ergodic Theory Dynam.\ Systems {\bf 19} (1999), 671--685.

\bibitem[LF$^+$]{Exbook} V.\ Lucarini, D.\ Faranda, A.C.\ Freitas, J.M.\ Freitas, M.\ Holland, T.\ Kuna, M.\ Nicol, M.\ Todd, S.\ Vaienti, Extremes in Dynamical Systems, Pure and Applied Mathematics: A Wiley Series of Texts, Monographs, and Tracts, 2016, pp 312.


\bibitem[MeS]{MSbook} W. de Melo, S. van Strien, {\it One
  dimensional  dynamics,} Ergebnisse Series {\bf 25}, Springer--Verlag, 1993.

\bibitem[PU]{PolUrb}  M.~Pollicott, M.~Urbanski, \emph{Open Conformal Systems and
Perturbations of Transfer Operators}, to appear in Lecture Notes in Mathematics {\bf 2206}, Springer.

\bibitem[PR]{PrzRiv11} F.~Przytycki, J.~Rivera-Letelier,
\emph{ Nice inducing schemes and the thermodynamics of rational maps,}
 Comm. Math. Phys. \textbf{301} (2011) 661--707.


\bibitem[R]{Rychlik}  M.~Rychlik, \emph{Bounded variation and invariant measures},  
Studia Math. {\bf LXXVI} (1983), 69--80. 

\bibitem[S1]{Sar99} O.~Sarig, \emph{Thermodynamic Formalism for Countable Markov Shifts}, Ergodic Theory Dyn. Syst. {\bf 19} (1999), 1565--1593.

\bibitem[S2]{Sar01} O.\ Sarig,
{\em Phase transitions for countable topological Markov shifts,}
Commun.\ Math.\ Phys.\ \textbf{217} (2001), 555--577.

\bibitem[S3]{Sar03} O.~Sarig, \emph{Existence of Gibbs measures for countable Markov shifts}, Proc. Amer. Math. Soc. {\bf 131} (2003), 1751--1758.



\bibitem[Y]{young poly} L.-S.\ Young, \emph{Recurrence times and rates of mixing},
Israel J. Math. {\bf 110} (1999), 153--188.

\bibitem[Yu]{Yur05} M.\ Yuri,
\emph{ Large deviations for countable to one Markov systems,}
 Comm. Math. Phys. \textbf{258} (2005), 455--474.

\end{thebibliography}
\end{document}